\documentclass[11pt,a4paper]{amsart}
\usepackage[T1]{fontenc}
\usepackage{lmodern,amsmath,amssymb,mathrsfs,mathtools,booktabs,enumitem,microtype}
\usepackage[margin=28mm]{geometry}
\usepackage{xcolor}
\usepackage[colorlinks=true,linkcolor=blue!50!black,citecolor=blue!50!black,urlcolor=blue!50!black]{hyperref}
\newtheorem{theorem}{Theorem}[section]
\newtheorem{proposition}[theorem]{Proposition}
\newtheorem{lemma}[theorem]{Lemma}
\newtheorem{corollary}[theorem]{Corollary}
\theoremstyle{definition}

\theoremstyle{remark}
\newtheorem{remark}[theorem]{Remark}
\newcommand{\A}{\mathcal A}

\newcommand{\Hh}{\mathfrak H}
\newcommand{\C}{\mathbb C}
\newcommand{\Q}{\mathbb Q}
\newcommand{\Z}{\mathbb Z}
\newcommand{\Ff}{\mathbb F}
\newcommand{\OO}{\mathcal O}

\newcommand{\MM}{\mathrm M}
\newcommand{\JJ}{\mathrm J}
\newcommand{\F}{F_{14}}
\DeclareMathOperator{\ord}{ord}
\DeclareMathOperator{\rk}{rank}
\DeclareMathOperator{\Sym}{Sym}
\DeclareMathOperator{\Sp}{Sp}
\DeclareMathOperator{\GL}{GL}
\DeclareMathOperator{\Pic}{Pic}
\DeclareMathOperator{\NS}{NS}
\DeclareMathOperator{\divi}{div}
\DeclareMathOperator{\Sing}{Sing}
\DeclareMathOperator{\Supp}{Supp}
\DeclareMathOperator{\Spec}{Spec}
\DeclareMathOperator{\tr}{tr}
\DeclareMathOperator{\Nm}{Nm}
\DeclareMathOperator{\length}{length}

\DeclareMathOperator{\Proj}{Proj}
\DeclareMathOperator{\Eff}{Eff}

\newcommand{\roundup}[1]{\left\lceil #1\right\rceil}
\newcommand{\rounddown}[1]{\left\lfloor #1\right\rfloor}
\title[The Kodaira dimension of $\mathcal A_6$]{The Kodaira dimension and minimal models of $\mathcal A_6$}
\author{Haowu Wang}
\address{School of Mathematics and Statistics, Wuhan University, Wuhan 430072, Hubei, China}
\email{haowu.wangmath@whu.edu.cn}
\date{October 29, 2026}
\hypersetup{
 pdftitle={The Kodaira dimension and minimal models of A6},
 pdfauthor={Haowu Wang},
 pdfsubject={Siegel modular forms, Kodaira dimension, and minimal models},
 pdfkeywords={A6, Siegel modular forms, Prym varieties, Kodaira dimension}
}
\begin{document}
\begin{abstract}
We prove that the moduli space $\mathcal A_6$ of principally polarized
abelian sixfolds has Kodaira dimension zero, determining the remaining
case in the Kodaira classification of the spaces $\mathcal A_g$.
More precisely, every nonzero scalar Siegel modular form of genus six,
weight $k$ and positive boundary order $b$ satisfies $k\ge7b$,
with equality precisely for powers of the weight-fourteen harmonic
theta series of Dittmann, Salvati Manni and Scheithauer.
The canonical ring is therefore a polynomial ring on one generator
of degree two. We also prove numerical dimension zero and construct
the unique projective $\Q$-factorial terminal minimal model. Its
Picard number is one and its canonical divisor has order two. The perfect cone compactification
is a Mori dream space, and the birational automorphism group of
$\mathcal A_6$ is trivial. The proof combines normal vanishing on
the Prym locus, polar intersections on abelian boundary fibers,
and an exact theta calculation. A discrepancy estimate recovers
the modular open subset from the minimal model.
\end{abstract}
\setlength{\skip\footins}{7pt plus 11pt minus 3pt}
\maketitle
\raggedbottom
\setcounter{tocdepth}{1}
\tableofcontents
\section{Introduction}
\label{sec:introduction}

Let $\A_g$ denote the coarse moduli space of principally polarized
complex abelian varieties of dimension $g$. Analytically,
\[
 \A_g=\Sp_{2g}(\Z)\backslash\Hh_g,
 \qquad \dim\A_g=\frac{g(g+1)}2,
\]
where $\Hh_g$ is the Siegel upper half-space. By the Kodaira
dimension of $\A_g$ we mean that of a smooth projective model of
its function field. This differs from the logarithmic Kodaira
dimension of the open moduli space.

The birational classification of these spaces is closely related
to the construction and vanishing of Siegel modular forms.
The spaces $\A_g$ are unirational for $g\le5$; the last two
dimensions were treated by constructions with curves and intermediate
Jacobians \cite{Clemens83,Donagi84,Verra84}. Loughran and Sankaran
have recently proved unirationality over $\Q$ for every $g\le5$
\cite[Theorem~1.3]{LS25}. Tai proved that
$\A_g$ is of general type for $g\ge9$ \cite{Tai82}.
Freitag obtained the case $g=8$ \cite{FreitagBook83}, and Mumford
obtained the case $g=7$ \cite{Mumford83}. The purpose of this paper
is to determine the remaining Kodaira dimension and to describe
some resulting features of the birational geometry of $\A_6$.

Write $\lambda_g$ for the Hodge class and $D_g$ for the boundary
divisor of the perfect cone compactification $\A_g^{\mathrm{Perf}}$.
A nonzero scalar form of weight $k$ and boundary order $b$ has an
interior zero divisor whose closure has class
$k\lambda_g-b[D_g]$. Its slope, for $b>0$, is $k/b$.
In genus six the canonical class is
\begin{equation}\label{eq:intro-canonical-class}
 K_{\A_6^{\mathrm{Perf}}}=7\lambda_6-[D_6].
\end{equation}
Thus the position of slope seven in the effective cone is central
to the Kodaira-dimension problem.

\subsection{Previous results}
The last decade has brought two complementary approaches to this
problem. Farkas and Verra proved that the universal abelian variety
over $\A_5$ is unirational \cite[Theorem~0.1]{FV16}. This gives
unirationality of the rank-one boundary of $\A_6$. Their construction
also produces curves covering that boundary, and the Hodge and
boundary degrees of these curves give
\[
 s_{\mathrm{Eff}}(\A_6)\ge\frac{53}{10}
\]
\cite[Theorem~0.4]{FV16}.
Alexeev, Donagi, Farkas, Izadi and Ortega constructed a generically
finite dominant Prym--Tyurin map from a Hurwitz space to $\A_6$
\cite{ADFIO20}. It uses degree-$27$ covers of $\mathbf P^1$ with
monodromy $W(E_6)$ and $24$ reflection branch points. In subsequent
work they computed the Hodge classes associated with the irreducible
representations of $W(E_6)$ and bounded the interior ramification
locus by the union of a Weyl--Petri divisor and the divisor where
$h^0(C,\pi^*\OO_{\mathbf P^1}(1))\ge3$
\cite[Theorems~1.3 and~1.4]{ADFIO22}. Here $\pi:C\to\mathbf P^1$
is the degree-$27$ cover. They also proved generic unramifiedness
along the boundary divisor mapping onto $D_6$. These results give a detailed parametrization by curves;
the ramification still has to be accounted for when passing to the
canonical geometry of $\A_6$.

On the modular side, Dittmann, Salvati Manni and Scheithauer
constructed a harmonic theta series $\F$ associated with the Leech
lattice \cite{DSS21}. It is a scalar cusp form of genus six and
weight $14$, with boundary order exactly two. It extends to a
nonzero bicanonical form on a smooth projective model of $\A_6$.
Consequently
\[
 \kappa(\A_6)\ge0,
 \qquad s_{\mathrm{Eff}}(\A_6)\le7.
\]
Their form is the starting point of the present paper. We do not
require its full Fourier expansion.
Grushevsky, Ibukiyama, Mondello and Salvati Manni developed
differential operators to study moving slopes \cite{GIMS24}.
They proved that the interior divisor of $\F$ is prime and reduced
and obtained the conditional bound
$s_{\mathrm{Mov}}(\A_6)\le43/6$, subject to a nondegeneracy
condition on its matrix of first derivatives
\cite[Corollary~E]{GIMS24}. We verify this
condition geometrically below.

Ma related Siegel modular forms to pluricanonical forms on Kuga
families and extended this correspondence to logarithmic
pluricanonical forms on suitable compactifications
\cite[Theorems~1.1 and~1.2]{Ma21}.
Poon, Salvati Manni and Sankaran determined the Kodaira dimensions
of compactified Kuga varieties with a positive number of marked
points \cite[Theorem~1]{PSS25}. In particular, the $n$-fold Kuga variety over
$\A_6$ has Kodaira dimension $21$ for every $n\ge1$. Their result
leaves the unmarked case considered here. Geometric constructions
of modular forms through vanishing conditions also occur in the
work of van der Geer and Kouvidakis on projectivized Hodge bundles
\cite{vGK24}. Our argument uses higher normal derivatives along a
Prym locus and their specialization to an abelian boundary fiber.
The construction of vector-valued Siegel modular forms by normal
derivatives along a product locus was developed by Cl\'ery and
van der Geer \cite{CV15}. We use this construction to determine the first possible normal
term of $\F$ along $\A_5\times\A_1$, and to control its
vanishing at the finite quotient singularities.

\subsection{The sharp slope bound and the canonical ring}
All scalar modular forms in this paper have trivial character.
We write $\MM_k^{(g)}$ for their space in genus $g$ and weight $k$,
and $b(G)$ for the boundary order of a nonzero form. These
conventions are specified in Section~\ref{sec:preliminaries}.

\begin{theorem}\label{thm:intro-slope}
Let $G$ be a nonzero scalar Siegel modular form for $\Sp_{12}(\Z)$,
of weight $k$ and positive boundary order $b$. Then
\[
 k\ge7b.
\]
Equality holds if and only if $b$ is even and
$G=c\F^{b/2}$ for a nonzero constant $c$.
\end{theorem}

The assertion is uniform in the weight and the boundary order.
In particular, it excludes every smaller slope and proves rigidity
on the canonical ray, rather than only in a finite range of weights.

\begin{theorem}\label{thm:intro-canonical}
For every smooth projective model $X$ of $\A_6$,
\[
 R(X,K_X)=\bigoplus_{m\ge0}H^0(X,mK_X)
       \simeq\C[u],\qquad \deg u=2,
\]
where $u$ is defined by $\F$. Moreover,
\[
 \kappa(\A_6)=\nu_\sigma(K_X)=0.
\]
\end{theorem}

Here $\nu_\sigma$ denotes Nakayama's numerical dimension, recalled
in Section~\ref{sec:applications}. Numerical dimension zero requires more than the description of the
canonical ring. For integers $\ell\ge0$ we consider all spaces with fixed
weight excess:
\[
 \mathcal M_\ell=
 \bigoplus_{m\ge0}\MM_{7m+\ell}^{(6)}[m],\qquad
 \MM_k^{(6)}[b]=\{G\in\MM_k^{(6)}:b(G)\ge b\}.
\]
Theorem~\ref{thm:finite-module} proves that $\mathcal M_\ell$
is a finitely generated free graded $\C[\F]$-module, where
$\deg\F=2$. Its homogeneous dimensions are therefore bounded.
This gives the section-growth estimate needed for
$\nu_\sigma(K_X)=0$.

\subsection{Minimal models and birational automorphisms}
Put $X_0=\A_6^{\mathrm{Perf}}$, and let $E_{\F}$ be the closure
of the interior zero divisor of $\F$. Thus
\[
 [E_{\F}]=14\lambda_6-2[D_6],\qquad 2K_{X_0}\sim E_{\F}.
\]
The preceding section-growth estimate permits the use of the
minimal model program in numerical dimension zero. It leads to
the following description.

\begin{theorem}\label{thm:intro-minimal}
The following statements hold.
\begin{enumerate}[label=\textup{(\roman*)},leftmargin=*]
\item There is a unique projective $\Q$-factorial terminal minimal
model $Y$ up to isomorphism. It has Picard number one and
\[
 2K_Y\sim0,\qquad K_Y\not\sim0.
\]
The valuation defined by the interior zero divisor $E_{\F}$ has
discrepancy $1/2$ over $Y$. In particular, no smooth projective
minimal model exists.
\item The perfect cone compactification $\A_6^{\mathrm{Perf}}$
is a Mori dream space. Its effective slope is seven, and its
moving slope is an attained rational number $\sigma$ satisfying
\[
 7<\sigma\le\frac{43}{6}.
\]
\item The strict transform $D_Y$ of the modular boundary is a
prime ample unirational $\Q$-Cartier divisor, and
\[
 Y\setminus D_Y\cong\A_6\setminus V(\F).
\]
\end{enumerate}
\end{theorem}

The strict inequality between the effective and moving slopes
describes the fixed canonical divisor and the last nef boundary
on the minimal model program. It also controls the ring containing
all prescribed boundary orders: Theorem~\ref{thm:boundary-algebra}
proves finite generation of
\[
 \bigoplus_{k,b\ge0}\MM_k^{(6)}[b]t^b
 \quad\text{and of}\quad
 \bigoplus_{k,b\ge0}
       \MM_k^{(6)}[b]/\MM_k^{(6)}[b+1].
\]
Finite generation of the first algebra answers, over $\C$ in genus
six, the question posed by Shepherd-Barron for the bigraded algebra
of forms with prescribed boundary order
\cite[\S2, concluding remark]{SB06}.
The second is the algebra of leading Fourier--Jacobi coefficients
which extend to genus-six forms with the indicated vanishing.
There is no assertion of finite generation for the full
bounded-relative-index rings considered in \cite{BBHJ24}.

The nef cone of the perfect cone compactification and the contraction
associated with $12\lambda_g-[D_g]$ were studied by Shepherd-Barron
\cite{SB06,SB23}. Over $\C$, the class is semiample for
$g\le11$, and the exceptional locus is the closure of
$\A_1\times\A_{g-1}$ \cite{SB23}.
Our use of the minimal model program determines the other endpoint
of the movable cone and proves that finitely many nef cones fill
the whole movable cone. Finite termination alone would not
establish this assertion.

The triviality of the birational automorphism group requires
additional information about the singularities of $Y$.
Freitag proved triviality for $\A_g$ when $g$ is sufficiently
large, using differential tensors \cite{Fre84}. Tsuyumine obtained
the explicit range $g\ge10$, together with a statement for
principal congruence subgroups, by proving that every divisor
in $\A_g$ is of general type \cite[Corollary]{Tsu87}.
The assertion below strengthens the description of the minimal
model to a statement about all birational self-maps.

\begin{theorem}\label{thm:intro-birational}
The valuation defined by $E_{\F}$ is the unique exceptional
divisorial valuation over $Y$ of discrepancy less than one:
\[
 a(E_{\F},Y)=\frac12,\qquad a(v,Y)\ge1\quad(v\ne\ord_{E_{\F}}).
\]
Moreover,
\[
 \operatorname{Bir}(\A_6)=\operatorname{Aut}(Y)=\{1\}.
\]
\end{theorem}

Here $a(v,Y)$ denotes discrepancy, so that the associated log
discrepancy is $1+a(v,Y)$. The intrinsic characterization of
$E_{\F}$ allows us to recover its extraction and then the modular
open subset from $Y$. Preservation of that open subset is a
conclusion of the argument, rather than a hypothesis on a
birational self-map.

\subsection{The proof and its geometric ingredients}
\subsubsection{Normal vanishing along the Prym locus}
The first difficulty is to obtain a lower bound that applies to all
weights. The genus-seven Prym map
\[
 P_7:\mathcal R_7\longrightarrow\A_6
\]
has image $\mathcal P_7$ of codimension three. A general Prym curve
of genus seven is a section of a standard Nikulin $K3$ surface
\cite{FV12}. Along a general pencil on this surface, we identify
the actual conormal bundle of the Prym image with
$\OO_{\mathbf P^1}(2)^{\oplus3}$. The Hodge and boundary degrees
are $6$ and $44$. A nonzero normal jet of order $j$ of a
weight-$k$ form of boundary order $b$ therefore satisfies
\[
 6k-44b+2j\ge0,
\]
and hence
\begin{equation}\label{eq:intro-normal}
 \ord_{\mathcal P_7}(G)\ge22b-3k.
\end{equation}
The extension across the boundary is essential in this calculation.
We first identify the Hodge bundle and the fixed space of quadrics
on the Nikulin surface. The global logarithmic codifferential then
induces a generically invertible map of rank-$18$ bundles on the
pencil. The canonical-class formula of Farkas--Ludwig and the
intersection numbers of Farkas--Verra show that its determinant
has degree zero. The map is therefore invertible everywhere,
including at the ramified boundary. This identifies the conormal
bundle itself and controls all its symmetric powers.
Nikulin surfaces also enter Farkas--Verra's proof of the
uniruledness of $\mathcal R_9$ \cite{FV24}. More recent work on
secant loci of Prym curves further studies the geometry of
Nikulin sections \cite{FLC25}.

\subsubsection{Polar intersections on the boundary}
The second difficulty is to turn normal vanishing on a subvariety
of codimension three into a divisorial constraint. The map
$P_6:\mathcal R_6\to\A_5$ has degree $27$ \cite{DS81}.
For a very general abelian fivefold $(A,\Theta)$ this gives
$27$ distinct Abel--Prym surfaces $\Sigma_i$ through the origin,
with
\[
 [\Sigma_i]=4\frac{\Theta^3}{3!}.
\]
The first nonzero Fourier--Jacobi coefficient of $G$ is a section
of $2b\Theta$ whose symbolic order on each $\Sigma_i$ is at least
$22b-3k$. We check the degree of the reduced surface and the
primitive extension coordinate in the Prym degeneration; either
normalization affects the ensuing intersection inequality.

Put $\phi=\phi_2(\F)$. It is an even section of $4\Theta$ and
has order at least two on every $\Sigma_i$. Elliptic modular
transformations force its order at the origin to be at least ten.
An exact theta-jet calculation proves that its initial form $T_{10}$
has a smooth plane section. This is the only computer-assisted
step: two explicitly specified matrices are checked over
$\Ff_{101}((t))$ to justified finite precision. A maximal-cofactor
argument transfers nonvanishing to characteristic zero. The
ancillary programs reconstruct all entries from finite integer sums
and include an independent verification.

The tangent-cone calculation proves local integrality and normality.
Polar intersections then show that $Z=\divi(\phi)$ is globally
integral and has generic order exactly two along every $\Sigma_i$.
These statements suffice for the slope and canonical-ring proofs;
the finer assertion about transverse quadratic singularities is
proved afterwards. More generally, for every $\alpha\in\Pic^0(A)$ and
every nonzero section
\[
 s\in H^0(A,\OO_A(d\Theta)\otimes\alpha)
\]
of symbolic order at least $e$ on all $27$ surfaces,
Theorem~\ref{thm:fixed-part} gives
\begin{equation}\label{eq:intro-fixed}
 \max\left\{0,\left\lceil\frac{9e-4d}{2}\right\rceil\right\}
 \le\ord_Z(s)\le\left\lfloor\frac d4\right\rfloor.
\end{equation}
The fixed polar divisor on the normalization is removed as a Weil
cycle before taking intersections; global normality of $Z$ is
not assumed. Formula~\eqref{eq:intro-fixed} implies $d\ge2e$,
with equality only for a scalar power of $\phi$. Substituting
$d=2b$ and $e=22b-3k$ proves $k\ge7b$.

\subsubsection{Equality and the canonical ring}
For equality, fiberwise proportionality still has to extend over
the base. We prove that $\phi$ has no vertical factor over
$\A_5$, using Freitag's exact-divisor theorem, the genus-five slope
$54/7$ from the corrected version of \cite{FGSMV14},
a finite-image norm estimate, and
$S_8(\Sp_{10}(\Z))=0$
\cite{Freitag83,FGSMV14,BR15,PY07}. This makes the quotient of
leading coefficients a holomorphic weight-zero modular function,
and hence a constant. The norm estimate is proved directly in
Section~\ref{sec:modular}. The correction \cite{BR25} to
\cite{BR15}, including its discussion of the incorrect
finite-generation statement, is recorded there; no corrected
convergence theorem for formal Fourier--Jacobi series is an input
to our proof.

Once the leading coefficient is a constant multiple of a power
of $\phi$, subtraction of the corresponding power of $\F$
increases the boundary order without changing the weight.
The strict slope inequality makes the difference zero. This
proves rigidity on the entire canonical ray. A pluricanonical
section determines a form of weight $7m$ and boundary order at
least $m$, whereas the known bicanonical section defined by
$\F$ extends to a smooth model. These two facts identify the
whole canonical ring.

\subsubsection{Numerical dimension and the minimal model}
The next step requires control of forms lying near the canonical
ray, not just on it. Write $k=7b+\ell$, with $\ell$ fixed.
Then \eqref{eq:intro-fixed} forces the leading coefficient to
be divisible by
\[
 \phi^{\lceil(b-27\ell)/2\rceil}\qquad(b>27\ell).
\]
After division, the remaining Jacobi weight and index depend
only on $\ell$ and a parity. Filtering by the finitely many
possible first boundary indices therefore bounds
$\dim\MM_{7m+\ell}^{(6)}[m]$ independently of $m$.
Multiplication by $\F$ is injective, so these dimensions stabilize
in each parity. This gives the finite free modules described
above. Notice that division of an entire Siegel modular form
is obtained from stabilization; it is not inferred merely from
division of its first Fourier--Jacobi coefficient.

On a resolution $X$, twisted pluricanonical sections inject into
these spaces. A big multiple of the Hodge bundle dominates any
fixed ample twist, so the same boundedness proves
$\nu_\sigma(K_X)=0$. Existence and abundance in numerical
dimension zero then give a good minimal model. Starting from
$X_0$, the relation $2K_{X_0}\sim E_{\F}$ forces every
$K$-negative divisorial contraction to contract $E_{\F}$.
Since $\rho(X_0)=2$, the program consists of finitely many flips
followed by one divisorial contraction, and its target has
Picard number one. The bicanonical section has a simple zero
along $E_{\F}$ and no zeros on the minimal model. Comparing
these orders gives $a(E_{\F},Y)=1/2$ and excludes a smooth
minimal model.

\subsubsection{The moving slope}
There is also a more precise local geometric consequence of the
tangent-cone argument. The $27$ surfaces have distinct degree-ten
Prym-canonical tangent curves. A complete intersection of $T_{10}$
with two general polars has degree
\[
 10\cdot9^2=810=3\cdot27\cdot10.
\]
If every transverse Hessian were degenerate, the contributions of
these curves would already use the entire degree. The Gauss map
forces an additional curve meeting the smooth locus, which is
impossible. Thus $\phi$, and then $\F$, has a nondegenerate
quadratic normal term along the general Prym locus
(Theorem~\ref{thm:prym-transverse-A1}).

This local result also proves the moving-slope bound without a
second theta calculation. The conormal space at a general Prym
point is the net of quadrics containing the corresponding Nikulin
surface in $\mathbf P^5$. The Nikulin lattice excludes quadrics
of rank at most four, so the discriminant of this net is a smooth
plane sextic. The nondegenerate normal quadratic form gives a
smooth conic in the same plane, which cannot be contained in
that sextic. It follows that the matrix of first derivatives of
$\F$ has full rank at some smooth point of its zero divisor.
The operator of \cite{GIMS24} therefore produces a form of weight
$86$ and boundary order $12$ not divisible by $\F$.
Together with $\F$, it gives a moving pencil of slope $43/6$.

\subsubsection{Recovering the modular open subset}
To control birational automorphisms, we combine discrepancies
on toroidal charts with the vanishing of the bicanonical form.
An elementary bound for the Reid--Tai ages leaves only one
possible small-age stabilizer, coming from an elliptic factor
with $j$-invariant zero. The order-ten vanishing of $\F$ along
the elliptic product locus raises the discrepancy of every
exceptional valuation in this remaining case to at least one. Minkowski's
convex body theorem supplies the required bound for new toric
divisors. These estimates distinguish the valuation of $E_{\F}$
intrinsically, and permit the recovery of the modular open subset
from an arbitrary automorphism of $Y$.
More precisely, the exceptional valuation determines its extraction
by the relative Proj of the valuation ideals. Lifting an automorphism
and reversing the flips gives a pseudoautomorphism of $X_0$ that
fixes the divisor classes. An ample class makes it regular, and
the Hodge class recovers the Satake contraction and its boundary.
The remaining assertion is the triviality of the automorphism group
of the open modular variety, obtained from the normalizer of
$\Sp_{12}(\Z)$ in the automorphism group of $\Hh_6$.

\subsection{Organization}
Section~\ref{sec:preliminaries} fixes the modular and geometric
normalizations. Section~\ref{sec:prym} proves the estimate for
normal vanishing and its boundary specialization.
Section~\ref{sec:theta} proves normality of the first
Fourier--Jacobi divisor near the origin; its finite certificate is
given in Appendix~\ref{sec:theta-certificate}.
Section~\ref{sec:polars} proves the polar inequalities and the
fixed-component theorem. Section~\ref{sec:modular} completes the
slope and canonical-ring arguments.
Section~\ref{sec:applications} proves numerical dimension zero
and constructs a minimal model.
Section~\ref{sec:singularities} determines the generic transverse
ordinary quadratic singularities.
Section~\ref{sec:moving} verifies the differential nondegeneracy
condition needed for the moving-slope bound.
Section~\ref{sec:birational} proves the statements about
divisor cones, finite generation and the boundary on the unique
minimal model. Section~\ref{sec:automorphisms} identifies its
distinguished exceptional valuation and proves the triviality of
the birational automorphism group.

\section{Modular forms, boundary orders and canonical forms}
\label{sec:preliminaries}

\subsection{Siegel modular forms}
Put
\[
 \Hh_g=\{\tau\in\operatorname{Mat}_{g\times g}(\C):
              \tau={}^{t}\tau,\ \operatorname{Im}\tau>0\},
 \qquad \Gamma_g=\Sp_{2g}(\Z).
\]
An element $\gamma=\left(\begin{smallmatrix}A&B\\ C&D\end{smallmatrix}\right)$
of $\Gamma_g$ acts by
$\gamma\tau=(A\tau+B)(C\tau+D)^{-1}$.
For an integer $k$, a scalar Siegel modular form of genus $g$ and
weight $k$ is a holomorphic function $G:\Hh_g\to\C$ satisfying
\[
 G(\gamma\tau)=\det(C\tau+D)^kG(\tau).
\]
For $g=1$ one also requires holomorphy at the cusp. For $g>1$ this
condition follows from the Koecher principle. We write
$\MM_k^{(g)}=\MM_k(\Gamma_g)$ for this finite-dimensional vector space.
Only the full group and the trivial character occur in scalar modular
forms in this paper.

Let $\mathcal T_g$ be the set of positive semidefinite half-integral
symmetric matrices: $T\in\mathcal T_g$ means that $2T$ is integral with
even diagonal. With $e(x)=\exp(2\pi i x)$, the Fourier expansion is
\[
 G(\tau)=\sum_{T\in\mathcal T_g}a_G(T)e(\tr(T\tau)),
 \qquad
 a_G({}^{t}UTU)=\det(U)^k a_G(T)\quad(U\in\GL_g(\Z)).
\]
For $T\in\mathcal T_g$ put
$m(T)=\min_{0\ne v\in\Z^g}{}^{t}vTv$. For a nonzero modular form define
\begin{equation}\label{eq:boundary-order}
 b(G)=\min_{a_G(T)\ne0}m(T).
\end{equation}
It is a nonnegative integer. The form is cuspidal if and only if
$b(G)>0$. The $\GL_g(\Z)$ symmetry shows that $b(G)$ is also the
smallest last diagonal entry $T_{gg}$ among its nonzero Fourier
coefficients.

For $b\ge0$ write
\[
 \MM_k^{(g)}[b]=\{G\in\MM_k^{(g)}:b(G)\ge b\},
\]
where $b(0)=+\infty$. This is a decreasing filtration, and
$b(G_1G_2)=b(G_1)+b(G_2)$ for nonzero scalar forms. Additivity can
be seen directly from the first nonzero Fourier--Jacobi coefficients,
whose product is a nonzero holomorphic function.

\subsection{Fourier--Jacobi coefficients and theta decomposition}
Write a point of $\Hh_{g+1}$ as
\[
 \Omega=\begin{pmatrix}\tau&z\\{}^{t}z&\omega\end{pmatrix},
 \qquad \tau\in\Hh_g,\quad z\in\C^g,\quad q=e(\omega).
\]
The Fourier--Jacobi expansion of $G\in\MM_k^{(g+1)}$ is
\begin{equation}\label{eq:fj-expansion}
 G(\Omega)=\sum_{m\ge0}\phi_m(G)(\tau,z)q^m.
\end{equation}
For $m>0$, the coefficient belongs to the space $\JJ_{k,m}^{(g)}$
of scalar Jacobi forms of genus $g$, weight $k$ and index $m$.
Our normalization is
\begin{align}
 \phi(\tau,z+\tau u+v)
  &=e\bigl(-m({}^{t}u\tau u+2{}^{t}uz)\bigr)\phi(\tau,z),
       &&u,v\in\Z^g,                                      \label{eq:jacobi-translation}\\
 \phi(\gamma\tau,(C\tau+D)^{-t}z)
  &=\det(C\tau+D)^k
    e\bigl(m\,{}^{t}z(C\tau+D)^{-1}Cz\bigr)\phi(\tau,z).
                                                               \label{eq:jacobi-modular}
\end{align}
The Fourier expansion is supported on pairs $(N,r)$ for which
$\left(\begin{smallmatrix}N&r/2\\{}^{t}r/2&m\end{smallmatrix}\right)$
is positive semidefinite. These conventions imply that, on the abelian
variety $A_\tau=\C^g/(\Z^g+\tau\Z^g)$, the function
$z\mapsto\phi(\tau,z)$ represents a section of $\OO_{A_\tau}(2m\Theta)$.
The symmetric theta divisor $\Theta$ is chosen with the usual
rigidification; the even tensor powers used here are independent of
the ambiguity relevant to this choice.

For $\mu\in(\Z/2m\Z)^g$ put
\[
 \theta_{m,\mu}(\tau,z)=
 \sum_{\substack{r\in\Z^g\\r\equiv\mu\ (2m)}}
 e\left(\frac{{}^{t}r\tau r}{4m}+{}^{t}rz\right).
\]
Theta decomposition gives
\begin{equation}\label{eq:theta-decomposition}
 \phi(\tau,z)=\sum_{\mu\in(\Z/2m\Z)^g}
                    h_\mu(\tau)\theta_{m,\mu}(\tau,z).
\end{equation}
The vector $h=(h_\mu)_\mu$ is a Siegel modular form of weight
$k-\frac12$ for a finite-image Weil representation of the metaplectic
group. We use the theta decomposition and the Koecher principle for
Jacobi forms in the form established in \cite{Ziegler89}; see also
\cite[Section~2.2]{BR15}. In particular, a holomorphic function on
$\Hh_g\times\C^g$ satisfying the Jacobi transformation laws with
positive index is a holomorphic Jacobi form when $g>1$.

For a Jacobi form, its boundary order in the $\tau$ variables is the
minimum of $N_{gg}$ among its nonzero coefficients, equivalently the
minimum after integral changes of basis. Suppose that this order is
at least $c$. In \eqref{eq:theta-decomposition}, choose a representative
with $-m<r_g\le m$. Comparing the coefficient with this fixed $r$ gives
\begin{equation}\label{eq:theta-order}
 \ord(h_\mu)\ge c-\frac{r_g^2}{4m}\ge c-\frac m4.
\end{equation}
The comparison is coefficientwise: translation
\eqref{eq:jacobi-translation} permits the chosen representative, so no
estimate on cancellation between different theta components is needed.

\subsection{The Hodge bundle and compactifications}
We use $\mathscr A_g$ for the moduli stack of principally polarized
abelian varieties and $\A_g=\Gamma_g\backslash\Hh_g$ for its coarse
moduli space. Let $\mathbb E_g$ be the Hodge bundle on the stack, and
put
\[
 \mathcal L_g=\det\mathbb E_g,\qquad
 \lambda_g=c_1(\mathcal L_g).
\]
Scalar forms of weight $k$ are sections of $\mathcal L_g^k$.
Bundle calculations can be made after passing to a sufficiently high
level, and the resulting equivariant statements descend. We retain
the full-level notation for scalar forms throughout.

The rank-one partial toroidal compactification is denoted by
$\mathscr A_g^\partial$; its boundary is $D_g$. The boundary over
$\mathscr A_{g-1}$ is the universal Kummer variety. On its double cover
by the universal abelian variety, a normal parameter is $q=e(\omega)$
as in \eqref{eq:fj-expansion}. Consequently
\[
 b(G)=\ord_{D_g}(G),\qquad
 G\in H^0(\mathscr A_g^\partial,\mathcal L_g^k(-bD_g))
       \quad\text{if }b(G)\ge b.
\]
At a general boundary point the involution $z\mapsto-z$ is free.
Passing between the abelian fiber and its Kummer quotient therefore
does not change a generic vanishing order.

The perfect cone compactification is denoted by $\A_g^{\mathrm{Perf}}$.
For $g\ge2$ we use the standard facts that its rank-one boundary has a single
irreducible closure, still denoted $D_g$, and that
\[
 N^1(\A_g^{\mathrm{Perf}})_\Q
       =\Q\lambda_g\oplus\Q[D_g].
\]
The classes $\lambda_g$ and $D_g$ are $\Q$-Cartier. A sufficiently
divisible positive power of $\mathcal L_g$ descends from an ample
line bundle on the Satake compactification. Its pullback to a smooth
projective resolution is therefore big and nef.
For these compactifications and their divisor classes we refer to
\cite{FC90,SB06}; the slope conventions agree with \cite{GIMS24}.

We shall also use the following exact form of the relation between
divisors and modular forms. For $g\ge3$, every nonzero effective
divisor on $\A_g$ is the zero divisor of a scalar modular form for
$\Gamma_g$ of positive integral weight, with precisely the prescribed
multiplicities \cite[Satz~2]{Freitag83}. No multiple of the divisor
and no additional character is required. The absence of divisorial
ramification for the effective modular action in these genera
identifies divisors on the coarse space with invariant divisors
upstairs. In genus five the weight is necessarily even, by the action
of $-I$. We will use this theorem both to divide out an interior
prime factor and to test a possible common divisor of theta
coefficients.

\subsection{Canonical forms and the distinguished cusp form}
Let $X$ be a smooth projective resolution of $\A_6^{\mathrm{Perf}}$.
It admits a morphism to the Satake compactification. The product
$\bigwedge_{1\le i\le j\le6}d\Omega_{ij}$ transforms with weight $-7$.
Near a general rank-one boundary point it has a simple logarithmic
pole, since $d\omega=(2\pi i)^{-1}dq/q$. Thus a pluricanonical section
necessarily determines a scalar form of weight $7m$ and order at least
$m$ at $D_6$.

For later use choose an integer $r>0$ such that $\mathcal L_6^r$
descends to an ample line bundle on the Satake compactification,
and denote its pullback to $X$ by $L_X$. For all integers $m,a\ge0$
restriction to the open modular variety gives an injection
\begin{equation}\label{eq:canonical-injection}
 H^0(X,mK_X+aL_X)\ \lhook\joinrel\longrightarrow\
                 \MM_{7m+ra}^{(6)}[m].
\end{equation}
Indeed the transformation law gives the weight $7m+ra$, and regularity
at a general rank-one boundary point gives the order $m$. The resolution
is an isomorphism at the generic point of every interior divisor, and
the genus-six quotient has no divisorial branch locus
\cite[Proposition~1.3 and the Reid--Shepherd-Barron--Tai criterion]{SB16}.
The omitted
subsets of $\Hh_6$ therefore have codimension at least two. Holomorphic
extension across them follows from normality;
the Koecher principle gives a holomorphic modular form. This argument
only gives a necessary condition and does not assert extension of
an arbitrary form on the right across every exceptional divisor.

The converse needed for one particular form is already known.
We recall its construction to specify the modular input.
Let $\Lambda$ be the Leech lattice, with its positive definite
bilinear form $\langle\, ,\,\rangle$. It is even and unimodular,
has rank $24$, and has minimum norm $4$. For six-tuples $x,y$ in
$\Lambda\otimes\C$, put
\[
 B(x,y)=(\langle x_i,y_j\rangle)_{1\le i,j\le6},
 \qquad B(x)=B(x,x).
\]
Choose a six-tuple $h$ with $B(h)=0$ and $B(h,\overline h)>0$.
The harmonic theta construction gives
\[
 F_h(\Omega)=
 \sum_{x\in\Lambda^6}\det B(x,h)^2\,
                   e\!\left(\frac12\tr(B(x)\Omega)\right).
\]
It transforms with weight $24/2+2=14$. A summand with
$\det B(x,h)\ne0$ has linearly independent $x_i$, so $B(x)$ is
positive definite. The minimum norm of $\Lambda$ gives boundary
order at least two. Dittmann, Salvati Manni and Scheithauer choose
$h$ explicitly and compute a nonzero coefficient of boundary
minimum two \cite[Sections~2 and~4]{DSS21}. Their result and the
extension theorem give the following statement.
\begin{theorem}[Dittmann--Salvati Manni--Scheithauer]\label{thm:known-F}
There is a nonzero cusp form $\F\in\MM_{14}^{(6)}$ such that
$b(\F)=2$. It is a harmonic theta series associated with the Leech
lattice. The corresponding bicanonical form extends to every smooth
projective model of $\A_6$.
\end{theorem}
This is \cite[Theorems~4.2 and~5.2, Proposition~5.1]{DSS21};
the extension uses Tai's theorem \cite{Tai82}.
We fix one normalization of $\F$ and put
\begin{equation}\label{eq:phi-definition}
 \phi=\phi_2(\F)\in\JJ_{14,2}^{(5)}.
\end{equation}
In particular $\phi\ne0$, $\phi(\tau,-z)=\phi(\tau,z)$, and its
restriction to a general fiber is a nonzero section of $4\Theta$.
We write $E_{\F}$ for the closure in $\A_6^{\mathrm{Perf}}$
of the interior zero divisor of $\F$.

\subsection{Symbolic vanishing}
If $V$ is an integral subvariety of a smooth variety $A$, and $s$ is
a section of a line bundle, $\ord_V(s)$ is its order in the maximal
ideal of the regular local ring $\OO_{A,\eta_V}$, after trivializing
the line bundle. For distinct integral subvarieties $\Sigma_i$ define
\[
 \mathcal I^{(e)}=\bigcap_i\mathcal I_{\Sigma_i}^{(e)}
 \quad(e\ge0).
\]
Here each symbolic power is the contraction of the $e$th power at the
generic point, and $\mathcal I^{(0)}=\OO_A$. Thus membership in
$H^0(A,M\otimes\mathcal I^{(e)})$ means precisely that the order of
the section is at least $e$ at the generic point of every $\Sigma_i$.
All orders used in the intersection arguments below have this meaning.

\section{Normal vanishing along the Prym locus}
\label{sec:prym}

We write $\mathscr R_g$ for the stack of connected \'etale double covers
of smooth curves of genus $g$, and $\mathcal R_g$ for its coarse moduli
space. A geometric point is represented by a pair $(C,\eta)$, where
$\eta\in\Pic^0(C)[2]\setminus\{\OO_C\}$. The associated Prym variety is
the identity component of the kernel of the norm map,
\[
 P(C,\eta)=\ker\bigl(\Nm:J(\widetilde C)\longrightarrow J(C)\bigr)^0,
\]
endowed with its principal polarization. It has dimension $g-1$.
The Prym morphism is denoted by
\[
 P_g:\mathscr R_g\longrightarrow\mathscr A_{g-1}.
\]
We use the same notation for the induced map of coarse spaces.
The generic Torelli theorem for the Prym map implies that
\[
 \mathcal P_7:=\overline{P_7(\mathcal R_7)}\subset\mathcal A_6
\]
is an integral subvariety of dimension $18$ and codimension $3$
\cite{FS82}.

We shall use two boundary divisors in the compactification by
quasistable Prym curves. A quasistable Prym curve consists of
$(C,\eta,\beta)$, where $\eta$ has total degree zero, its restriction
to every exceptional rational bridge $E$ is $\OO_E(1)$, and
$\beta:\eta^{\otimes2}\to\OO_C$ is an isomorphism away from the
exceptional components and their attachment points. A general point of
$\Delta'_0$ has an irreducible one-nodal stable curve whose normalization
carries a nontrivial two-torsion line bundle. A general point of
$\Delta^{\mathrm{ram}}_0$ has an exceptional bridge attached to a smooth
curve at two points. We refer to \cite[Section~1]{FL10} for this
compactification.

Let $\mathscr R_7^\circ$ be the open partial compactification containing
the smooth locus and these two boundary divisors, with higher boundary
strata removed. We restrict further whenever necessary to the open
set where the Prym morphism is defined. It maps to the rank-one partial
compactification $\mathscr A_6^\partial$. On this open set one has
\begin{equation}\label{eq:prym-boundary-pullback}
 P_7^*D_6=\Delta'_0,
 \qquad
 P_7(\Delta^{\mathrm{ram}}_0)\subset\mathscr A_6;
\end{equation}
the first equality includes the multiplicity of the pullback divisor
\cite[Section~6, equation~(9)]{FGSMV14}.

For a nonzero modular form $G$, the order
$\ord_{\mathcal P_7}(G)$ means its order in the maximal ideal of the
regular local ring at the generic point of $\mathcal P_7$. It is
equivalently the first nonzero normal Taylor degree along the smooth
part of the Prym locus.

\begin{theorem}\label{thm:prym-order}
Let $G\in\MM_k(\Sp_{12}(\Z))$ be nonzero and let
$b=\ord_{D_6}(G)>0$. Then
\begin{equation}\label{eq:prym-normal-bound}
 \ord_{\mathcal P_7}(G)\geq 22b-3k.
\end{equation}
\end{theorem}

The proof requires the actual conormal bundle along a complete
family of Prym curves. We first construct its fixed kernel using
the Prym Hodge bundle. A determinant calculation then excludes
degeneracy of the codifferential at the boundary.

\subsection{Nikulin pencils and multiplication of differentials}

Let $S$ be a general standard Nikulin surface. Thus $S$ is a $K3$
surface with eight disjoint smooth rational curves $N_1,\ldots,N_8$,
and a line bundle $e$ such that
\[
 2e=N_1+\cdots+N_8.
\]
We choose a nef class $H$ satisfying
\[
 H^2=12,\qquad H\cdot N_i=0,
 \qquad \Pic(S)=\Z H\mathbin{\perp}\mathfrak N,
\]
where $\mathfrak N$ is the Nikulin lattice generated by the classes of
the $N_i$ and $e$. Put $L=H-e$. The complete linear system $|L|$
embeds $S$ as a smooth complete intersection of three quadrics in
$\mathbf P^5$. In particular,
\begin{equation}\label{eq:nikulin-section-dimensions}
 h^0(S,L)=6,\qquad h^0(S,2L)=18,
 \qquad \dim I_2(S)=3,
\end{equation}
where $I_2(S)=H^0(\mathbf P^5,\mathcal I_S(2))$.
The lattice polarization is described in \cite[Definition~1.2]{FV12};
the embedding and its quadrics follow from
\cite[Lemma~1.3(iii)--(iv) and Lemma~2.1]{FV12}.

A general pencil $B\cong\mathbf P^1$ in $|H|$ has $12$ base points.
Let
\[
 f:X_S\longrightarrow B,\qquad \sigma:X_S\longrightarrow S
\]
be its incidence family and the blowup morphism, respectively. The
Prym data on this family are
\begin{equation}\label{eq:nikulin-prym-data}
 \eta=\sigma^*\OO_S(-e),\qquad
 \beta:\eta^{\otimes2}=\OO_{X_S}\left(-\sum_{i=1}^8E_i\right)
       \longrightarrow\OO_{X_S},
\end{equation}
where $E_i$ is the proper transform of $N_i$. In particular,
$\eta|_{E_i}=\OO_{E_i}(1)$. On a smooth member of $|H|$, the
restrictions of $e$ and $-e$ are isomorphic, but the sign in
\eqref{eq:nikulin-prym-data} matters on the exceptional bridges.

There are eight fibers of the form $E_i\cup C_i$, with
$C_i$ smooth of genus $6$ and $E_i\cdot C_i=2$. The other singular
fibers are irreducible and nodal. The pencil is transverse to the
corresponding open boundary divisors. It defines a morphism
$\gamma:B\to\mathscr R_7^\circ$. Write $\lambda_7$ for the Hodge
class of the underlying stable genus-seven curves. The intersection
numbers are
\begin{equation}\label{eq:nikulin-pencil-boundaries}
 B\cdot\lambda_7=8,\qquad B\cdot\Delta'_0=44,\qquad
 B\cdot\Delta^{\mathrm{ram}}_0=8.
\end{equation}
For example, the genus-seven $K3$ pencil has total stable boundary
degree $60$, and each of the eight ramified points contributes twice
to that degree. These are also the intersection numbers of
\cite[Proposition~1.4]{FV12}.

Write $\omega_f$ for the relative dualizing sheaf and set
\[
 \mathcal E=f_*(\omega_f\otimes\eta).
\]
The relative adjunction formula for the incidence divisor in
$S\times B$ gives $\omega_f=\OO_{X_S}(H,1)$, where
$\OO_{S\times B}(M,r)$ denotes
$\operatorname{pr}_S^*\OO_S(M)\otimes
\operatorname{pr}_B^*\OO_B(r)$.
The restriction sequence
\[
 0\longrightarrow\OO_{S\times B}(-e,0)
 \longrightarrow\OO_{S\times B}(L,1)
 \longrightarrow\omega_f\otimes\eta\longrightarrow0
\]
and the vanishing $H^i(S,-e)=0$ give
\begin{equation}\label{eq:nikulin-pushforwards}
 \mathcal E=H^0(S,L)\otimes\OO_B(1).
\end{equation}
The required vanishing follows from
\cite[Lemma~1.3(i)]{FV12} and Serre duality.

\subsection{The Hodge extension and the conormal bundle}

\begin{lemma}\label{lem:nikulin-hodge}
The vector bundle $\mathcal E$ is the pullback of the extended
Prym Hodge bundle. In particular, for $h=P_7\circ\gamma$,
\[
 h^*\mathbb E_6\cong\OO_B(1)^{\oplus6},\qquad
 \deg h^*\mathcal L_6=6.
\]
\end{lemma}

\begin{proof}
Let $\rho:\widehat S\to S$ be the double cover branched along
$\sum_iN_i$. The surface $\widehat S$ is the blowup of a $K3$
surface at the eight fixed points of a symplectic involution.
The trace decomposition is
\[
 \rho_*\omega_{\widehat S}
       =\omega_S\oplus\omega_S(e),
\]
where the second summand is anti-invariant. Relative adjunction
for the pencil of divisors $\rho^*H$ gives the anti-invariant
Hodge bundle
\[
 H^0(S,H+e)\otimes\OO_B(1);
\]
here $H^1(\widehat S,\omega_{\widehat S})=0$, and the
anti-invariant part of $H^0(\widehat S,\omega_{\widehat S})$
is zero. Since $(H+e)\cdot N_i=-1$, the divisor
$\sum_iN_i=2e$ is a fixed part of $|H+e|$. Removing it
gives the fixed identification
$H^0(S,H+e)=H^0(S,H-e)=H^0(S,L)$.

At each ramified parameter the lifted pencil contains twice
an exceptional $(-1)$-curve. Contracting this curve gives
an ordinary nodal fiber in the pencil on the covering $K3$
surface. Pushforward of the relative dualizing sheaf is unchanged
by the blowup. Thus the bundle just computed is the
anti-invariant part of the stable-cover Hodge extension, and
hence the Prym Hodge extension. The fixed identification above
identifies it with $\mathcal E$.
\end{proof}

\begin{proposition}\label{prop:nikulin-conormal}
The Prym morphism is immersive along the general pencil.
Its ordinary conormal bundle, pulled back to $B$, is
\begin{equation}\label{eq:nikulin-conormal}
 \mathcal K\cong I_2(S)\otimes\OO_B(2).
\end{equation}
More precisely, logarithmic Kodaira--Spencer duality identifies
the logarithmic codifferential with the exact sequence
\begin{equation}\label{eq:nikulin-multiplication}
 0\longrightarrow I_2(S)\otimes\OO_B(2)
 \longrightarrow\Sym^2\mathcal E
 \xrightarrow{\,dP_7^*\,}
 \gamma^*\Omega^1_{\mathscr R_7^\circ}(\log\Delta'_0)
 \longrightarrow0.
\end{equation}
\end{proposition}

\begin{proof}
Put
\[
 \mathcal T=
 \gamma^*\Omega^1_{\mathscr R_7^\circ}(\log\Delta'_0).
\]
The identity $P_7^*D_6=\Delta'_0$ gives a logarithmic
codifferential on the whole pencil. By
Lemma~\ref{lem:nikulin-hodge} and logarithmic
Kodaira--Spencer duality on $\mathscr A_6^\partial$
\cite{FC90}, it is a bundle homomorphism
\[
 \alpha:\Sym^2H^0(S,L)\otimes\OO_B(2)
                 \longrightarrow\mathcal T.
\]
At a smooth member $C$ of the pencil, this is the usual
multiplication map of Prym-canonical differentials, up to a
nonzero scalar. The restriction of $H^0(S,L)$ to $C$ is
$H^0(C,\omega_C\otimes\eta_C)$.
Moreover, $H^0(S,2L-H)=0$ by the proof of
\cite[Lemma~2.1]{FV12}. The quadrics containing $C$
therefore coincide with the quadrics containing $S$.
Thus the kernel of $\alpha$ on the smooth part of $B$
is the fixed subbundle $I_2(S)\otimes\OO_B(2)$.

This subbundle is consequently killed by $\alpha$ everywhere,
since its image is a holomorphic bundle homomorphism vanishing
on a dense open set. We obtain a homomorphism of bundles of rank
eighteen,
\[
 \overline\alpha:
 \bigl(\Sym^2H^0(S,L)/I_2(S)\bigr)\otimes\OO_B(2)
                \longrightarrow\mathcal T,
\]
which is an isomorphism over the smooth part of the pencil.

It remains to rule out zeros of its determinant at the boundary.
The canonical-class formula \cite[Theorem~1.5]{FL10}, restricted to our partial
compactification, is
\[
 K_{\mathscr R_7^\circ}
       =13\lambda_7-2\Delta'_0-3\Delta^{\mathrm{ram}}_0.
\]
This also follows directly from the Hurwitz formula:
the forgetful map has local smoothing parameter $u=t^2$
along $\Delta^{\mathrm{ram}}_0$ and is unramified along
$\Delta'_0$. There are no elliptic tails or other boundary
divisors in this open set. At the ramified fibers, the
nonexceptional subcurve is connected, so there is no
inessential automorphism \cite[Remark~6.3]{FL10};
the scalar deck automorphism acts trivially on the cotangent
space. The same formula therefore computes the degree of the
cotangent bundle pulled back by the actual Prym family.

Using $B\cdot\lambda_7=8$ and
\eqref{eq:nikulin-pencil-boundaries}
\cite[Proposition~1.4]{FV12}, we find
\[
 \deg\mathcal T
 =(K_{\mathscr R_7^\circ}+\Delta'_0)\cdot B
 =13\cdot8-44-3\cdot8=36.
\]
The source of $\overline\alpha$ also has degree $18\cdot2=36$.
Its determinant is therefore a nonzero section of a
degree-zero line bundle on $B\cong\mathbf P^1$, and has no zeros.
Hence $\overline\alpha$ is an isomorphism everywhere. This
proves \eqref{eq:nikulin-multiplication}.

Along $\Delta^{\mathrm{ram}}_0$, neither the source cotangent
bundle in \eqref{eq:nikulin-multiplication} nor the target
cotangent bundle has a logarithmic pole. Along $\Delta'_0$,
the equality $P_7^*q=t$ up to a unit identifies the residue
lines. Comparing the ordinary and logarithmic cotangent
sequences shows that the ordinary codifferential is also
surjective and has the same kernel. This proves both
immersion and \eqref{eq:nikulin-conormal}.
\end{proof}

\begin{proof}[Proof of Theorem~\ref{thm:prym-order}]
Put $j=\ord_{\mathcal P_7}(G)$. Choose a general interior point of
the Prym locus at which the $j$-th normal jet is nonzero. By
\cite[Theorem~0.4]{FV12}, a general genus-seven Prym curve is a section
of a unique general standard Nikulin surface. We may therefore choose
a general pencil through the corresponding curve on that surface.

Regard $G$ as a holomorphic section of
$\mathcal L_6^{\otimes k}(-bD_6)$ on the partial compactification.
In coordinates where an immersed branch has ideal
$I=(y_1,y_2,y_3)$ and $q$ is a boundary coordinate along the
branch, $\OO/I^j$ is $q$-torsion free. Thus $I^j$ is saturated
with respect to $q$, and division by the prescribed boundary
factor preserves the lower normal-jet vanishing.
On each immersed Prym branch, its normal jets of orders less than
$j$ vanish on the dense interior. Since they are holomorphic, they
also vanish at the boundary points of the branch. Its $j$-th normal
jet thus defines a section of
\begin{align*}
 h^*\bigl(\mathcal L_6^{\otimes k}(-bD_6)\bigr)
       \otimes\Sym^j\mathcal K
 &\cong\Sym^j I_2(S)\otimes
        \OO_B(6k-44b+2j).
\end{align*}
It is nonzero at the point through which the pencil was chosen.
Hence $6k-44b+2j\geq0$, which is precisely
\eqref{eq:prym-normal-bound}.

This construction uses normal jets on the source branches. It does
not assume that the global image $\mathcal P_7$ is smooth at every
point of the pencil. The local normal jets glue through the intrinsic
conormal bundle. The family is an actual family of Prym data on $B$;
the scalar deck action is trivial on $\Sym^2\mathcal E$ and
$\det\mathcal E$. At a general ramified fiber the nonexceptional
subcurve is connected, so there is no additional inessential
automorphism \cite[Remark~6.3]{FL10}. Thus no further degree factor
is introduced by passage between the stack and the coarse space.
\end{proof}

\subsection{Abel--Prym surfaces in the boundary}

Fix a general principally polarized abelian fivefold $(A,\Theta)$.
The map $P_6:\mathcal R_6\to\mathcal A_5$ is generically \'etale
of degree $27$ \cite{DS81}. Let
\[
 \pi_i:\widetilde C_i\longrightarrow C_i,
 \qquad 1\leq i\leq27,
\]
be its Prym presentations, and let $\iota_i$ denote the involution.
Under a Prym identification with $A$, define
\begin{equation}\label{eq:abel-prym-surface}
 \Sigma_i=
 \left\{\OO_{\widetilde C_i}
    (p+q-\iota_i p-\iota_i q):p,q\in\widetilde C_i\right\}
 \subset A
\end{equation}
with its reduced structure. This is an integral surface. It contains
the origin and is invariant under inversion, so the ambiguity in a
general Prym identification does not affect it.

\begin{proposition}\label{prop:prym-boundary-surfaces}
The surfaces $\Sigma_1,\ldots,\Sigma_{27}$ are pairwise distinct and
satisfy
\begin{equation}\label{eq:prym-surface-class}
 [\Sigma_i]=4\frac{\Theta^3}{3!},\qquad
 \Theta^2\cdot\Sigma_i=80.
\end{equation}
They are the reduced extension surfaces obtained from the general
points of $\Delta'_0$ over $A$ in the rank-one boundary of
$\mathcal A_6^\partial$.
\end{proposition}

\begin{proof}
The map $\delta_2:\widetilde C_i^{(2)}\to\Sigma_i$ in
\eqref{eq:abel-prym-surface} has generic degree one, since $2<6/2$
\cite[Appendix~A, Corollary~A.15]{LZ22}. Its pushforward class is
$4\Theta^3/3!$ by \cite[Appendix~A, Proposition~A.11]{LZ22}.
Thus this is the class of the reduced image, not a multiple of it.
Since $\Theta^5=5!$, the intersection number is $80$.

The support of the projectivized tangent cone of $\Sigma_i$ at the
origin is the Prym-canonical image of $C_i$ in $\mathbf P T_0A$
\cite[p.~4]{Izadi99}. For a general cover this is an embedding.
Indeed, \cite[Appendix~A, Corollary~A.10]{LZ22}, applied with
$d=2<6/2$, gives
$\eta_i\notin C_i^{(2)}-C_i^{(2)}$. Thus, for every effective
divisor $T$ of degree two on $C_i$, Riemann--Roch gives
\[
 h^0\bigl(C_i,\omega_{C_i}\otimes\eta_i(-T)\bigr)
 =3+h^0\bigl(C_i,\eta_i(T)\bigr)=3.
\]
The five-dimensional space $H^0(C_i,\omega_{C_i}\otimes\eta_i)$
therefore separates every subscheme of length two.
Consequently the embedded surface recovers the abstract curve $C_i$
and the line bundle
$\eta_i=\OO_{C_i}(1)\otimes \omega_{C_i}^{-1}$.
Two coincident surfaces would therefore give the same point of
$\mathcal R_6$, contrary to the choice of the $27$ presentations.

We check the extension coordinate, since its normalization affects
both the class and the vanishing order. At a general point of
$\Delta'_0$, the upper stable curve is obtained by identifying
$p_1$ with $q_1$ and $\iota p_1$ with $\iota q_1$ in a smooth
cover $\widetilde C\to C$. Its generalized Jacobian has torus
$\mathbf G_m^2$, with norm $(s_1,s_2)\mapsto s_1s_2$.
The Prym torus has character lattice
\[
 X^*(T_P)=\Z^2/\Z(1,1).
\]
Its primitive generator is the class of the first standard basis
vector. The difference of the two standard basis vectors is twice
this generator and must not be used instead.

Let $i:A\hookrightarrow J(\widetilde C)$ be the Prym inclusion and
$a=\OO_{\widetilde C}(p_1-q_1)$. The extension line bundle is
$i^*\lambda_J(a)$. The polarization identity
\[
 i^*\lambda_J=\lambda_A\circ(1-\iota)
\]
identifies its point in $A$ with
\[
 z=(1-\iota)(p_1-q_1)
   =\delta_2(p_1+\iota q_1).
\]
Its reduced image is precisely \eqref{eq:abel-prym-surface}.
In particular, there is no multiplication-by-two map in the
extension parameter. The boundary smoothing parameter is primitive
by \eqref{eq:prym-boundary-pullback}. Equivalently, the two vanishing
cycles upstairs are exchanged by $\iota$; their difference is
primitive in the Prym lattice, on which the principal polarization
is half the restricted Jacobian polarization. Picard--Lefschetz
monodromy then has coefficient one, giving $q=t$ up to a unit.
\end{proof}

\begin{corollary}\label{cor:prym-leading-order}
Let $G$, $k$ and $b$ be as in Theorem~\ref{thm:prym-order}, and let
$\psi_b$ be its first nonzero Fourier--Jacobi coefficient. On a
general abelian fivefold one has
\begin{equation}\label{eq:boundary-symbolic-order}
 \ord_{\Sigma_i}(\psi_b)\geq\ord_{\mathcal P_7}(G)\geq22b-3k
 \qquad(1\leq i\leq27).
\end{equation}
Here the order is the symbolic order at the generic point of the
reduced surface. For a fixed globally defined Jacobi form, these
$27$ generic orders have a common value on a sufficiently general
abelian fivefold.
\end{corollary}

\begin{proof}
We first verify immersion at the general points of $\Delta'_0$,
independently of the pencils used above. Locally, its parameters
are a genus-six Prym cover, two marked points, and a discrete gluing
choice. The map $P_6$ is \'etale at the general point, and
$\delta_2$ is generically birational in characteristic zero, hence
immersive at its general point. The boundary map has rank
$15+2=17$. The smoothing parameter satisfies $q=t$ up to a unit,
so the full Prym map has rank $18$ and is transverse to $D_6$.

In a local target chart, write the immersed branch as
$(y_1,y_2,y_3)=0$, with $q$ a transverse boundary parameter.
Set $j=\ord_{\mathcal P_7}(G)$. As in the proof of
Theorem~\ref{thm:prym-order}, the holomorphic section $G/q^b$
belongs to $(y_1,y_2,y_3)^j$ along the branch. Restriction first
to $q=0$ and then to a fixed general point of $\mathcal A_5$
gives
\[
 \psi_b\in\mathcal I_{\Sigma_i}^{(j)}.
\]
Equation~\eqref{eq:boundary-symbolic-order} now follows from
\eqref{eq:prym-normal-bound}. The Kummer quotient causes no loss
of order, since inversion acts freely at the general point of
$\Sigma_i$ and the quotient is \'etale there.

For the final assertion, consider the irreducible family of
Abel--Prym surfaces over the irreducible space $\mathcal R_6$.
It is dominated by the relative symmetric square of the connected
universal cover. On a smooth finite level chart, restrict to the
relative smooth locus of this family in the universal abelian scheme;
this is a regular embedding of codimension three. Let $r$ be the
order of the Jacobi form at its generic point. The normal jets of
degrees less than $r$ vanish identically, whereas the jet of degree
$r$ is a regular section nonzero at the generic point. Its
nonvanishing locus has dense constructible image in the base, by
Chevalley's theorem. After shrinking the base, it meets each
integral fiber in a nonempty open subset. The order at the generic
point of every such fiber is consequently $r$. This gives an open
subset of $\mathcal R_6$ on which the order is constant. After
shrinking $\mathcal A_5$ so that $P_6$ is finite \'etale and deleting
the image of the remaining proper exceptional subset, all $27$
presentation points have the same generic order.
\end{proof}

\section{The first Fourier--Jacobi coefficient and its tangent cone}
\label{sec:theta}

Let $\phi=\phi_2(\F)$, where $\F$ is the form constructed in
\cite{DSS21}.  For $\tau\in\Hh_5$, we write
\[
 A_\tau=\C^5/(\Z^5+\tau\Z^5)
\]
and denote its principal polarization by $\Theta$.  The Jacobi transformation
law identifies $\phi(\tau,\cdot)$ with a section of
$\OO_{A_\tau}(4\Theta)$.  The purpose of this section is to prove the
following local property of its divisor. The finite calculation used in
the proof is given in Appendix~\ref{sec:theta-certificate}, by explicit
integer sums and linear algebra over $\Ff_{101}$.

\begin{proposition}\label{prop:theta-normal}
For a general principally polarized abelian fivefold $(A,\Theta)$, the
section $\phi$ has multiplicity exactly $10$ at the origin.  Its initial
homogeneous form is irreducible and reduced.  Moreover, the divisor
$\divi(\phi)$ is integral and normal in a neighborhood of the origin.
\end{proposition}

\subsection{Vanishing at the origin}

We first establish the vanishing that will provide an exact identity
for the theta-jet matrix.

\begin{lemma}\label{lem:origin-ten}
For every $\tau\in\Hh_5$, the section $\phi(\tau,\cdot)$ is even and
vanishes to order at least $10$ at the origin, unless it is identically zero.
\end{lemma}

\begin{proof}
Write a period matrix of degree six as
\[
 Z=\begin{pmatrix}\tau&z\\z^t&\omega\end{pmatrix},
 \qquad \tau\in\Hh_5,\quad z\in\C^5,\quad \omega\in\Hh_1.
\]
The transformation associated with $\operatorname{diag}(I_5,-1)$ shows
that $\F$ is even in $z$. Let $d$ be the least degree for which
the homogeneous Taylor term $f_d(\tau,z,\omega)$ is not
identically zero as $(\tau,\omega)$ varies.
By the normal derivative construction of
\cite[Proposition~2.1]{CV15}, the first nonzero Taylor term has
elliptic weight $14+d$. Indeed the conormal bundle of
$\A_5\times\A_1$ on the smooth moduli stack is
$\mathbb E_5\boxtimes\mathbb E_1$, and taking its $d$th
symmetric power adds $d$ to the elliptic weight.
Every coefficient of this homogeneous polynomial is an elliptic modular
form of weight $14+d$.  It has order at least $2$ at the cusp, because
the Fourier--Jacobi expansion of $\F$ starts with index $2$.
The elliptic valence formula implies $14+d\geq24$.  Thus $d\geq10$.
Taking the coefficient of $e^{4\pi i\omega}$ proves the assertion.
\end{proof}

\subsection{The finite calculation and normality}

Let $J(\tau)$ be the map taking an even section of $4\Theta$
to its Taylor coefficients of degrees $0,2,4,6,8$ at the origin.
Its matrix has $791$ rows and $528$ columns; the precise basis
and normalizations are given in Appendix~\ref{sec:theta-certificate}.

\begin{lemma}\label{lem:finite-theta}
There exists $\tau_0\in\Hh_5$ such that $J(\tau_0)$ has rank
$527$ and its kernel is generated by a section whose degree-ten
initial homogeneous form has a smooth plane section.
\end{lemma}
\begin{proof}
Appendix~\ref{sec:theta-certificate} proves the identically valid
rank bound \eqref{eq:jet-rank-upper}, and verifies nonzero maximal
minors of the jet matrix and of the degree-ten Macaulay matrix.
Lemma~\ref{lem:cofactor-lift} transfers these nonvanishing
statements to characteristic zero. The theta moments in
\eqref{eq:jet-integral-matrix} and \eqref{eq:tangent-moment}
converge for $|t|<1$, so their cofactors and Macaulay minors
converge as well. Nonzero formal minors are nonzero analytic
functions. Choosing $t_0\ne0$ sufficiently small proves the lemma.
\end{proof}

\begin{proof}[Proof of Proposition~\ref{prop:theta-normal}]
Apply Lemma~\ref{lem:finite-theta}.
In a neighborhood of the corresponding period matrix, the kernel of
$J(\tau)$ is a holomorphic line bundle.  Let $s(\tau)$ be a local
nonvanishing generator.  Then $\phi(\tau,\cdot)=a(\tau)s(\tau)$ for a
holomorphic function $a$.  Although $a$ could vanish on the entire
chosen one-parameter family, it cannot vanish on this neighborhood,
because $\phi$ is globally nonzero.  At nearby points where $a\ne0$,
the degree-ten initial form of $\phi$ therefore has a smooth plane
section.  This also proves that its multiplicity is exactly $10$.

These properties hold for a general moduli point.  To see this
algebraically, pass to a finite level cover carrying a universal
polarized abelian scheme and its symmetric theta bundle.  The bundle
of even sections, the jet map at the zero section, and the map taking
the leading homogeneous term are algebraic.  On the locus of rank
$527$, the kernel is a line bundle.  In its relative Grassmannian of
planes, the locus of smooth restrictions is open.  Its projection to
the base is open, since the Grassmannian projection is smooth.  The
calculation proves that this open set is nonempty.  Removing the
proper zero locus of $\phi$ and descending from the level cover gives
the assertion for a general $(A,\Theta)$.

Let $f$ be a local equation of $\divi(\phi)$ at the origin, and write
$T_{10}$ for its initial homogeneous form.  If $T_{10}$ were reducible
or nonreduced, its restriction to the selected plane would either be
zero or have two positive-degree factors.  In the latter case those
plane curves meet, by B\'ezout's theorem, and their intersection is
singular.  This contradicts the smoothness of the restriction.
Thus $T_{10}$ is irreducible and reduced, and the local equation $f$
is irreducible as well.

The singular locus of the hypersurface $T_{10}=0$ in $\mathbf P^4$ has
dimension at most $1$.  Indeed every projective subvariety of dimension
at least $2$ meets the selected plane, and an ambient singular point
on that plane would be a singular point of the restricted curve.
For each nonzero derivative of $T_{10}$ we have
\[
 \operatorname{in}\!\left(\frac{\partial f}{\partial z_i}\right)
       =\frac{\partial T_{10}}{\partial z_i}.
\]
The initial ideal of the local Jacobian ideal thus contains these
derivatives.  Their affine zero set has dimension at most $2$; Euler's
identity makes the equation $T_{10}=0$ redundant.  Comparison with the
associated graded ring gives
\[
 \dim_0\Sing\bigl(\divi(\phi)\bigr)\leq2.
\]
The divisor has dimension $4$ and is a hypersurface, so it satisfies
Serre's condition $S_2$.  The displayed inequality gives condition
$R_1$ near the origin.  Serre's normality criterion finishes the proof.
\end{proof}

\subsection{The normal term along the elliptic product locus}

The same calculation determines the normal term of $\F$ itself.
We state this on the smooth moduli stack, or equivalently on a
sufficiently small neat level cover. This convention matters:
the residual involution on the coarse space sends $z$ to $-z$.

\begin{corollary}\label{cor:product-normal-term}
Let $\mathcal D$ be the locus of products $A\times E$, where
$\dim A=5$ and $E$ is an elliptic curve. In period coordinates
\[
 Z=\begin{pmatrix}\tau&z\\z^t&\omega\end{pmatrix},
\]
let $T_{10}(\tau,z)$ be the degree-ten Taylor term of
$\phi(\tau,z)$ at $z=0$. Then
\begin{equation}\label{eq:product-normal-term}
 \F(Z)=\Delta(\omega)^2T_{10}(\tau,z)+O((z)^{12}),
\end{equation}
where $\Delta(\omega)=q+O(q^2)$ is the normalized elliptic
discriminant form. At a general $A$ and every elliptic factor
$E$, the order of $\F$ along $\mathcal D$ is exactly ten.
\end{corollary}
\begin{proof}
Lemma~\ref{lem:origin-ten} proves that all Taylor terms of $\F$
of degree less than ten vanish. The degree-ten coefficients
have elliptic weight $24$ and cusp order at least two.
Dividing by $\Delta^2$ gives a holomorphic elliptic modular
form of weight zero, hence a constant in $\omega$.
Comparing the coefficients of $q^2$ identifies the resulting
polynomial with $T_{10}(\tau,z)$. Evenness removes the odd
Taylor terms. Finally, $\Delta$ has no zeros on $\Hh_1$, and
Proposition~\ref{prop:theta-normal} gives $T_{10}(\tau,z)\ne0$
for general $\tau$.
\end{proof}

\section{Polar intersections and symbolic vanishing}
\label{sec:polars}

We first give an intersection argument on an abelian variety. It is
important that the divisor in this argument need only be normal at
specified generic points. Its normalization is used elsewhere, and the
fixed part of its polar system is treated as a Weil divisor.

\subsection{A local intersection inequality}
\begin{lemma}\label{lem:local-length}
Let $(R,\mathfrak m)$ be a regular local ring of dimension $c$. Suppose
that $f_1,\ldots,f_c$ generate an $\mathfrak m$-primary ideal and
$f_i\in\mathfrak m^{a_i}$, where $a_i\ge1$. Then
\[
 \length R/(f_1,\ldots,f_c)\ge a_1\cdots a_c.
\]
\end{lemma}
\begin{proof}
We recall the multiplicity argument to specify the local contribution
that will be used. If $(B,\mathfrak n)$ has dimension $d>0$, and
$f\in\mathfrak n^a$ lowers dimension by one, multiplication by $f$ gives
\[
 \length B/(\mathfrak n^{n+1},f)
 \ge
 \length B/\mathfrak n^{n+1}
       -\length B/\mathfrak n^{n+1-a}.
\]
Comparing the leading coefficients of the Hilbert--Samuel polynomials
gives $e(B/fB)\ge a\,e(B)$. The elements $f_1,\ldots,f_c$ form a
system of parameters in the Cohen--Macaulay ring $R$, and hence a
regular sequence. Apply the preceding inequality successively to their
images in the quotient rings. Since $e(R)=1$ and the final ring has
dimension zero, its multiplicity is its length. This proves the claim.
\end{proof}

\subsection{Removing the fixed polar divisor}
Let $A$ be an abelian variety of dimension $g\ge3$, and let $H$ be an
ample divisor on $A$. Let $E\subset A$ be an integral effective divisor
numerically equivalent to $aH$, with $a>0$. An invariant vector field
$v\in H^0(A,T_A)$ differentiates a local equation $f$ of $E$.
The restrictions $v(f)|_E$ define sections of $\OO_A(E)|_E$.
Indeed replacing $f$ by $uf$ gives
$v(uf)=u\,v(f)+v(u)f$, so the additional term vanishes on $E$.
We call these sections the polars of $E$.
Invariant vector fields span every tangent space of $A$, and therefore
their common zero scheme is the Jacobian singular scheme of $E$.

Let $\nu:\widetilde E\to E$ be the normalization and put
\[
 h=\nu^*(H|_E),\qquad P=\nu^*(\OO_A(E)|_E).
\]
Numerically $P=ah$. The pullback polar system is nonzero and has a
common fixed effective Weil divisor $\mathfrak B$ on $\widetilde E$.
Write the divisor of a general polar as
$\mathfrak B+R_v$. The residual system $R_v$ has no fixed prime
divisor. Every prime component of $\mathfrak B$ maps, under the finite
map $\nu$, onto a codimension-one subvariety of $E$ contained in
$\Sing(E)$.

\begin{proposition}\label{prop:polar}
Let $V_1,\ldots,V_N$ be distinct integral subvarieties of codimension
three in $A$. Suppose that $E$ is normal at their generic points, and
put $r_i=\ord_{V_i}(E)\ge2$. Then
\begin{equation}\label{eq:polar-self}
 \sum_{i=1}^N r_i(r_i-1)^2(H^{g-3}\cdot V_i)
       \le a^3H^g.
\end{equation}
If $D$ is another effective divisor, numerically equivalent to $dH$,
which does not contain $E$, and $\ord_{V_i}(D)\ge e_i\ge1$, then
\begin{equation}\label{eq:polar-mixed}
 \sum_{i=1}^N r_i(r_i-1)e_i(H^{g-3}\cdot V_i)
       \le a^2dH^g.
\end{equation}
\end{proposition}
\begin{proof}
At each generic point $\eta_{V_i}$ the normalization is an isomorphism.
Moreover $\mathfrak B$ does not contain that point: a codimension-one
singular component through it would contradict normality of the
localized hypersurface.

Choose $R_v$ containing no prime component of $\mathfrak B$. Then
choose a full Cartier polar divisor $P_w$ containing no component of
$R_v$. Such choices are possible because the residual system has no
fixed prime divisor, and only finitely many prime components have to be
avoided. Their intersection is an effective cycle of dimension $g-3$.
Intersecting it with $h^{g-3}$ gives
\begin{align*}
 (R_v\cdot P_w\cdot h^{g-3})
 &= (P^2\cdot h^{g-3})-(\mathfrak B\cdot P\cdot h^{g-3})\\
 &\le a^3H^g.
\end{align*}
The inequality follows because $P$ and $h$ are nef Cartier classes
and $\mathfrak B$ is an effective Weil cycle. This calculation does
not take the self-intersection of $\mathfrak B$ or of $R_v$.

In the regular local ring $\OO_{A,\eta_{V_i}}$, of dimension three,
the equations $f,v(f),w(f)$ have orders at least
$r_i,r_i-1,r_i-1$. The intersection just constructed is proper, so
these equations generate an ideal primary to the maximal ideal.
Near $\eta_{V_i}$ the residual polar is the full polar, because
$\mathfrak B$ is absent. Its cycle coefficient is consequently
\[
 \length\OO_{A,\eta_{V_i}}/(f,v(f),w(f))
      \ge r_i(r_i-1)^2
\]
by Lemma~\ref{lem:local-length}. Adding the contributions proves
\eqref{eq:polar-self}.

For the second assertion, restrict $D$ to $E$ and pull back to
$\widetilde E$. This is an effective Cartier divisor $D'$ because
$E$ is not a component of $D$. Choose a residual polar $R_v$
containing no component of $D'$. The proper intersection $D'\cdot R_v$
is effective, and
\[
 (D'\cdot R_v\cdot h^{g-3})
 = (D'\cdot P\cdot h^{g-3})
       -(D'\cdot\mathfrak B\cdot h^{g-3})
 \le a^2dH^g.
\]
Here $D'$ is numerically $dh$, so the subtracted term is nonnegative.
At $\eta_{V_i}$ the three local equations have orders at least
$r_i,e_i,r_i-1$. Lemma~\ref{lem:local-length} again computes the
required lower bound for the coefficient of $V_i$. This proves
\eqref{eq:polar-mixed}.
\end{proof}

\subsection{The divisor defined by the first Fourier--Jacobi coefficient}
We now fix a very general principally polarized abelian fivefold
$(A,\Theta)$. In particular $\NS(A)=\Z[\Theta]$.
Let $\Sigma_1,\ldots,\Sigma_{27}$ be the surfaces constructed in
Section~\ref{sec:prym}. Their properties needed here are
\begin{equation}\label{eq:surface-properties}
 0\in\Sigma_i,\qquad
 [\Sigma_i]=\frac23\Theta^3,\qquad
 \ord_{\Sigma_i}(\phi)\ge2.
\end{equation}
They are integral and distinct. Proposition~\ref{prop:theta-normal}
shows that $\divi(\phi)$ is integral and normal in a neighborhood
of the origin.

\begin{proposition}\label{prop:phi-integral}
The divisor $Z=\divi(\phi)$ on $A$ is integral and reduced.
It is normal at the generic point of every $\Sigma_i$, and
\[
 \ord_{\Sigma_i}(\phi)=2\qquad(1\le i\le27).
\]
\end{proposition}
\begin{proof}
Let $E$ be the unique global prime component of $\divi(\phi)$ through
the origin. Its multiplicity is one. Each integral surface $\Sigma_i$
is contained in $\divi(\phi)$ and contains the origin, so it is
contained in $E$ and in no other prime component. Hence
\[
 r_i:=\ord_{\Sigma_i}(E)=\ord_{\Sigma_i}(\phi)\ge2.
\]
The normal locus of $E$ contains a neighborhood of the origin.
Its intersection with $\Sigma_i$ is therefore nonempty, so it
contains the generic point of $\Sigma_i$.

Write $E\equiv a\Theta$. Since $\NS(A)=\Z[\Theta]$ and
$\divi(\phi)\in|4\Theta|$, one has $1\le a\le4$.
Proposition~\ref{prop:polar} gives
\[
 a^3\Theta^5\ge
 \sum_{i=1}^{27}r_i(r_i-1)^2(\Theta^2\cdot\Sigma_i)
 \ge36\Theta^5.
\]
Thus $a=4$. Every nonzero effective divisor has positive
$\Theta$-degree, so no further component or multiplicity can occur.
Consequently $\divi(\phi)=E$.

By the last assertion of Corollary~\ref{cor:prym-leading-order},
the integers $r_i$ have a common value $r$ on a sufficiently general
abelian fivefold. Applying the same inequality with $a=4$ yields
\[
 18r(r-1)^2\le64.
\]
Since $r\ge2$, this forces $r=2$. The normality assertion was
already established above.
\end{proof}

\subsection{Fixed components of the linear systems}
For integers $d,e\ge0$ and $\alpha\in\Pic^0(A)$ define
\begin{equation}\label{eq:W-definition}
 W_{d,e}(\alpha)=
 H^0\bigl(A,\OO_A(d\Theta)\otimes\alpha\otimes\mathcal I^{(e)}\bigr).
\end{equation}
We omit $\alpha$ when it is trivial.

\begin{theorem}\label{thm:fixed-part}
If $0\ne s\in W_{d,e}(\alpha)$ and $j=\ord_Z(s)$, then
\begin{equation}\label{eq:fixed-part}
 \max\left\{0,\roundup{\frac{9e-4d}{2}}\right\}
       \le j\le\rounddown{\frac d4}.
\end{equation}
In particular $W_{d,e}(\alpha)=0$ if $d<2e$. On the boundary
$d=2e$ one has
\begin{equation}\label{eq:critical-fiber}
 W_{2e,e}(\alpha)=
 \begin{cases}
   \C\,\phi^{e/2},& e\text{ even and }\alpha\simeq\OO_A,\\
   0,&\text{otherwise}.
 \end{cases}
\end{equation}
\end{theorem}
\begin{proof}
Remove the whole $Z$-component of $\divi(s)$:
\[
 s=\phi^j u,\qquad
 u\in H^0\bigl(A,\OO_A((d-4j)\Theta)\otimes\alpha\bigr),
 \qquad Z\not\subset\divi(u).
\]
Effectivity gives $d-4j\ge0$, proving the upper bound.
If $d-4j>0$ and $e-2j>0$, Proposition~\ref{prop:polar}
applied to $Z$ and $\divi(u)$ gives
\[
 36(e-2j)\le16(d-4j).
\]
If $e-2j\le0$, the same inequality holds trivially. If $d-4j=0$,
a nonzero section of the numerically trivial line bundle $\alpha$
has no zeros and trivializes $\alpha$, so again $e-2j\le0$.
Rearranging in all cases gives $2j\ge9e-4d$, as required.

Combining this inequality with $j\le d/4$ gives $d\ge2e$.
If $d=2e$, both bounds force $j=e/2$. Thus $e$ must be even.
The residual line bundle is $\alpha$ and its nonzero section
trivializes it. This proves \eqref{eq:critical-fiber}.
\end{proof}

\begin{corollary}\label{cor:fixed-deficit}
Let $0\le e<b$ and put
$J=\max\{0,\lceil(9e-8b)/2\rceil\}$. Then
\[
 W_{2b,e}=\phi^J W_{2(b-2J),\,e-2J}.
\]
If $\delta=b-e$ and $9e>8b$, and if $\varepsilon\in\{0,1\}$ is
congruent to $e$ modulo two, the residual indices are
\[
 b-2J=9\delta-\varepsilon,\qquad
 e-2J=8\delta-\varepsilon.
\]
\end{corollary}
\begin{proof}
Theorem~\ref{thm:fixed-part} gives the divisibility. Since
$\ord_{\Sigma_i}(\phi)=2$, division subtracts exactly $2J$ from
the order at each $\Sigma_i$ and yields the stated space.
The reverse inclusion follows by multiplication. The inequalities
$0\le e<b$ imply $J\le\lfloor e/2\rfloor$, so the residual orders
are nonnegative. The last formulas follow by writing
$2J=9e-8b+\varepsilon$.
\end{proof}

\section{Vertical factors and the minimal slope}
\label{sec:modular}

\subsection{A finite-image norm}
We recall a slope estimate for finite-dimensional representations with
finite image. This estimate is already
\cite[Proposition~2.4]{BR15}; its proof is included because the common
order of all components is essential in its application here.

The correction \cite{BR25} replaces the continuation argument in
\cite[Lemma~4.4 and Theorem~4.5]{BR15}. It also explains how to avoid
the incorrect finite-generation theorem of Runge quoted in
\cite[Theorem~3.8]{BR15}; see \cite[Remark~1.3]{BR25} and
\cite{BBHJ24}. Neither argument is used here. We work with the
Fourier--Jacobi coefficients of the holomorphic form supplied by
\cite{DSS21}, and the following multiplicative-norm proof is independent
of formal Fourier--Jacobi convergence and of finite generation of
rings of Jacobi forms.

\begin{lemma}\label{lem:finite-norm}
Let $h$ be a nonzero holomorphic Siegel modular form of genus five,
scalar weight $w\in\frac12\Z$, and finite-image metaplectic type.
If all its components have boundary order at least $\nu>0$, then
\[
 w\ge\frac{54}{7}\nu.
\]
There is no nonzero such form of negative weight.
\end{lemma}
\begin{proof}
Write $G_\rho$ for the finite image of the representation and
$N=|G_\rho|$. Choose $\tau_0$ where $h(\tau_0)\ne0$ and a linear
functional $\ell$ avoiding the finitely many hyperplanes
$\ell(Ah(\tau_0))=0$ for $A\in G_\rho$. Then
\[
 P(\tau)=\prod_{A\in G_\rho}\ell(Ah(\tau))
\]
is nonzero and holomorphic. Multiplication by a representation matrix
permutes the factors. Consequently $P$ is scalar of weight $Nw$,
and its boundary order is at least $N\nu$.
The element $(I,-1)$ in the kernel of the metaplectic cover gives
$(-1)^{2Nw}=1$, so $Nw$ is integral. The element $-I$ of
$\Sp_{10}(\Z)$ then forces $Nw$ to be even. Thus $P$ is a
full-level scalar Siegel modular
form, without an additional character.

The slope theorem for $\A_5$ gives
$Nw\ge(54/7)N\nu$; we use the corrected version of
\cite[Theorem~0.7]{FGSMV14}. The same norm construction, without an
order hypothesis, reduces negative-weight vanishing to the scalar
case. The latter follows from the Koecher principle and the
ampleness of the Hodge line bundle on the Satake compactification.
\end{proof}

\subsection{Primitivity over the base}
\begin{proposition}\label{prop:vertical-primitive}
Every nonzero Jacobi form in $\JJ_{14,2}^{(5)}$ whose boundary
order in the genus-five variables is at least two has no vertical
divisor on the universal abelian fivefold. Equivalently, its theta
coefficient vector has no common divisor on $\A_5$.
In particular this applies to $\phi=\phi_2(\F)$.
\end{proposition}
\begin{proof}
Let $\psi$ be such a Jacobi form. Its theta coefficient vector $h$
has weight $27/2$ and the Weil type associated with index two.
Applying \eqref{eq:theta-order} with $m=c=2$ shows
that every theta component has boundary order at least $3/2$.

Suppose that a prime divisor on the base divides all components.
By Freitag's exact-divisor theorem
\cite[Satz~2]{Freitag83}, it is the exact zero divisor of a
full-level scalar form $H$ of positive integral weight $d$.
Since the genus is five, $d$ is even. Put $c=b(H)$.
The quotient $h/H$ is holomorphic of finite Weil type, weight
$27/2-d$ and common order at least $3/2-c$.
Lemma~\ref{lem:finite-norm} first gives $d\le12$.
The scalar slope inequality $d\ge(54/7)c$ therefore gives $c\le1$.
In particular $3/2-c$ is positive, and another application of the
lemma gives
\begin{equation}\label{eq:vertical-numerical}
 0\le d-\frac{54}{7}c
 \le\frac{27}{2}-\frac32\frac{54}{7}
 =\frac{27}{14}.
\end{equation}
If $c=0$, no positive even $d$ satisfies this inequality.
If $c=1$, the only possibility is $d=8$. But
$S_8(\Sp_{10}(\Z))=0$ by \cite[Example~12]{PY07}.
This is again impossible.

To justify holomorphy of the quotient, work first on a smooth finite
level cover of the base. Divisibility by the base divisor means
componentwise divisibility of the theta expansion, hence division
produces a holomorphic finite-type form. The Koecher principle
provides its holomorphic Fourier expansion. The same statement
descends, since the divisor and the original Jacobi form are
full-level invariant. Thus no growth assumption about a meromorphic
quotient has been introduced in \eqref{eq:vertical-numerical}.
Finally, every diagonal entry in a nonzero Fourier coefficient of
$\F$ is at least two, so $\phi$ satisfies the required boundary-order
hypothesis.
\end{proof}

\begin{lemma}\label{lem:global-division}
Let $\psi$ be a holomorphic Jacobi form of genus five. Suppose that,
on a very general abelian fiber, $\phi^J$ divides $\psi$ for some
integer $J\ge0$. Then $\psi/\phi^J$ is holomorphic on
$\Hh_5\times\C^5$ and satisfies the quotient Jacobi transformation law.
If its index is positive it is a holomorphic Jacobi form.
If its weight and index are both zero, it is constant.
\end{lemma}
\begin{proof}
Pass to a smooth finite level base, so that the universal family is
an abelian scheme. Proposition~\ref{prop:phi-integral} implies that the
geometric generic fiber of $\divi(\phi)$ is integral and reduced.
There is therefore one horizontal prime component, with multiplicity
one, on each connected component of this level cover. Here one may
first restrict to the open base where the divisor is flat; openness
of geometric integrality for this proper family transfers the very
general fiber assertion to its geometric generic fiber. The assumed
generic divisibility removes the possible pole on this component.

A divisor in the total space cannot map into a subset of the base
of codimension at least two, because all fibers have dimension five.
Any remaining vertical divisor would thus dominate a base divisor,
and at its generic point would be the entire abelian fiber.
Proposition~\ref{prop:vertical-primitive} excludes precisely this
possibility. Indeed a vertical component after finite level change
would map onto a vertical component at full level. The quotient has
no divisorial poles. The total space
is smooth, so normality extends it across subsets of codimension at
least two.

Dividing the transformation laws gives the asserted automorphy.
For positive index the Jacobi Koecher principle applies.
For index zero the quotient restricts to a holomorphic function on
each compact abelian fiber, hence is independent of $z$. If its
weight is also zero, it is a holomorphic modular function on $\Hh_5$.
The Koecher principle extends it to the Satake compactification,
where it is constant.
\end{proof}

\subsection{The uniform slope theorem}
\begin{theorem}\label{thm:minimal-slope}
Let $0\ne G\in\MM_k^{(6)}$ and suppose $b=b(G)>0$. Then
\[
 k\ge7b.
\]
If $k=7b$, then $b$ is even and
$G=c\F^{b/2}$ for some $c\in\C^\times$.
\end{theorem}
\begin{proof}
Let $\psi=\phi_b(G)$ be the first nonzero Fourier--Jacobi coefficient.
By Corollary~\ref{cor:prym-leading-order}, its order on every
$\Sigma_i$ in a very general abelian fivefold is at least
\[
 e=22b-3k.
\]
If $k<7b$, then $e>b$, whereas $\psi$ is a section of $2b\Theta$.
The vanishing assertion of Theorem~\ref{thm:fixed-part} gives
$\psi=0$, a contradiction.

If $k=7b$, then $e=b$. Equation~\eqref{eq:critical-fiber} gives
that $b$ is even and that, on a very general fiber,
$\psi$ is a scalar multiple of $\phi^{b/2}$.
The quotient has weight zero and index zero.
Lemma~\ref{lem:global-division} shows that it is a single constant
$c$, independent of the fiber.
The modular form $G-c\F^{b/2}$ has weight $7b$ and boundary order
strictly larger than $b$. The strict slope inequality just proved
forces this difference to vanish.
\end{proof}

\begin{corollary}\label{cor:canonical-ring}
For every smooth projective model $X$ of $\A_6$,
\[
 R(X,K_X)=\bigoplus_{m\ge0}H^0(X,mK_X)
          \simeq\C[u],\qquad \deg u=2,
\]
where $u$ is the bicanonical section defined by $\F$.
In particular $\kappa(\A_6)=0$.
\end{corollary}
\begin{proof}
By birational invariance, choose $X$ as in
\eqref{eq:canonical-injection}. For $m=0$ the assertion is
$H^0(X,\OO_X)=\C$. For $m>0$, a nonzero section of $mK_X$ gives
a form of weight $7m$ and boundary order $b\ge m$.
Theorem~\ref{thm:minimal-slope} implies $b=m$, and then the form is
a scalar multiple of $\F^{m/2}$ with $m$ even.
Conversely all nonnegative powers of the bicanonical section
in Theorem~\ref{thm:known-F} extend to $X$. This proves the equality
of graded algebras.
\end{proof}

\begin{corollary}\label{cor:F-prime}
The zero divisor of $\F$ on $\A_6$ is prime and reduced.
\end{corollary}
\begin{proof}
Freitag's exact-divisor theorem realizes any prime component as the
zero divisor of a scalar modular form $H$. Dividing $\F$ by $H$
gives a holomorphic scalar form $G$ by the Koecher principle.
Negative-weight vanishing ensures that the weight of $G$ is
nonnegative.
For nonzero scalar forms of nonnegative weight put
$\epsilon(H)=\operatorname{wt}(H)-7b(H)$.
Theorem~\ref{thm:minimal-slope} and the weight-zero case give
$\epsilon(H),\epsilon(G)\ge0$, while their sum is
$14-7\cdot2=0$. Thus a nonconstant factor must itself have slope seven.
Its boundary order is then a positive even integer.
Since the total order of $\F$ is two, there can be only one
nonconstant factor, occurring with multiplicity one.
\end{proof}

\begin{remark}
Primality of $\F$ was already obtained in
\cite[Section~5.6]{GIMS24}. The preceding proof records that it also
follows from the uniform theorem, independently of low-weight
dimension tables.
\end{remark}

\begin{corollary}\label{cor:exact-orders}
The generic order of $\F$ on $\mathcal P_7$ is exactly two.
The boundary order of the Jacobi form $\phi$ in its genus-five
variables is also exactly two.
\end{corollary}
\begin{proof}
Put $j=\ord_{\mathcal P_7}(\F)$. Theorem~\ref{thm:prym-order}
gives $j\ge2$. The boundary restriction in the proof of
Corollary~\ref{cor:prym-leading-order} gives
\[
 j\le\ord_{\Sigma_i}(\phi)=2,
\]
where the equality follows from Proposition~\ref{prop:phi-integral}.
Hence $j=2$.

The second order is at least two by the Fourier coefficients of
$\F$. If it were at least three, \eqref{eq:theta-order} would give
common theta-vector order at least $5/2$. Lemma~\ref{lem:finite-norm}
would then imply $27/2\ge(54/7)(5/2)$, which is false.
\end{proof}

\section{Numerical dimension and minimal models}
\label{sec:applications}

The slope theorem determines the canonical ring. To obtain a minimal
model with numerically trivial canonical divisor, one needs an
additional argument: Kodaira dimension zero alone does not establish
numerical dimension zero. We prove the required estimate by allowing
a fixed positive difference between the weight and seven times the
boundary order.

\subsection{Jacobi forms obtained after division}
For an integer $\ell\ge1$ and $\varepsilon\in\{0,1\}$, define
$\mathcal B_{\ell,\varepsilon}$ to be the subspace of
\[
 \JJ_{190\ell-7\varepsilon,\,27\ell-\varepsilon}^{(5)}
\]
consisting of forms whose boundary order in the genus-five variables
is at least $27\ell-\varepsilon$, and whose order on each $\Sigma_i$
in a very general fiber is at least $24\ell-\varepsilon$.
These are finite-dimensional vector spaces. We also set
\[
 \mathcal B_{0,0}=\C,\qquad \mathcal B_{0,1}=0.
\]

\begin{proposition}\label{prop:bounded-leading}
Suppose that $0\ne G\in\MM_{7b+\ell}^{(6)}$ has boundary order $b$,
where $\ell\ge1$ and $b>27\ell$. Put
\[
 \varepsilon\equiv b-3\ell\pmod2,\quad
 0\le\varepsilon\le1,\qquad
 J=\roundup{\frac{b-27\ell}{2}}.
\]
Its leading Fourier--Jacobi coefficient has the form
\begin{equation}\label{eq:bounded-leading}
 \phi_b(G)=\phi^J\chi,\qquad
                 \chi\in\mathcal B_{\ell,\varepsilon}.
\end{equation}
\end{proposition}
\begin{proof}
The normal-order estimate gives symbolic order at least
$22b-3(7b+\ell)=b-3\ell$ on the $27$ surfaces.
Theorem~\ref{thm:fixed-part}, with $d=2b$ and $e=b-3\ell$,
therefore gives the stated power of $\phi$ on a very general fiber.
Lemma~\ref{lem:global-division} makes the quotient a holomorphic
Jacobi form. Since
\[
 2J=b-27\ell+\varepsilon,
\]
its weight, index and symbolic order are respectively
\begin{align*}
 7b+\ell-14J&=190\ell-7\varepsilon,\\
 b-2J&=27\ell-\varepsilon,\\
 b-3\ell-2J&=24\ell-\varepsilon.
\end{align*}
Finally the genus-five boundary order of $\phi_b(G)$ is at least
$b$, by the integral symmetry of the Fourier coefficients of $G$.
Corollary~\ref{cor:exact-orders} gives boundary order exactly two
for $\phi$. Additivity of this valuation in \eqref{eq:bounded-leading}
therefore gives the asserted lower bound $b-2J$ for $\chi$.
\end{proof}

\subsection{Finite modules over the canonical algebra}
For $\ell,m\ge0$ put
\[
 \mathcal M_{m,\ell}=\MM_{7m+\ell}^{(6)}[m],
 \qquad
 \mathcal M_\ell=\bigoplus_{m\ge0}\mathcal M_{m,\ell}.
\]
The direct sum is graded by $m$. Multiplication by $\F$ makes
$\mathcal M_\ell$ a graded module over $\C[\F]$, with $\deg\F=2$.

\begin{theorem}\label{thm:finite-module}
For every fixed $\ell\ge0$, the dimensions
$\dim\mathcal M_{m,\ell}$ are bounded independently of $m$.
More precisely, if $m>27\ell$ and
$\varepsilon\equiv m-\ell\pmod2$, $0\le\varepsilon\le1$, then
\begin{equation}\label{eq:module-bound}
 \dim\mathcal M_{m,\ell}
 \le \sum_{t=0}^{\lfloor\ell/7\rfloor}
                    \dim\mathcal B_{\ell-7t,\varepsilon}.
\end{equation}
The module $\mathcal M_\ell$ is a finitely generated free graded
$\C[\F]$-module. In particular its Hilbert function is eventually
periodic with period dividing two.
\end{theorem}
\begin{proof}
For a nonzero form of weight $7m+\ell$, Theorem~\ref{thm:minimal-slope}
bounds its first boundary index by
$m+\lfloor\ell/7\rfloor$. Filter $\mathcal M_{m,\ell}$ by boundary
orders
\[
 b=m+t,\qquad 0\le t\le\lfloor\ell/7\rfloor.
\]
The quotient between two consecutive terms injects, by its first
Fourier--Jacobi coefficient, into a Jacobi space. Its weight excess
relative to $b$ is $h=\ell-7t$.
If $h>0$, Proposition~\ref{prop:bounded-leading} applies because
$b\ge m>27\ell\ge27h$. It embeds this quotient into
$\mathcal B_{h,\varepsilon}$, since
\[
 b-3h=m+t-3(\ell-7t)\equiv m-\ell\pmod2.
\]
If $h=0$, Theorem~\ref{thm:minimal-slope} gives the same bound with
$\mathcal B_{0,0}=\C$ and $\mathcal B_{0,1}=0$.
Adding the dimensions of the successive quotients proves
\eqref{eq:module-bound}. There are only finitely many remaining
values of $m$, so boundedness holds for all $m$.

Multiplication by $\F$ gives injections
$\mathcal M_{m,\ell}\hookrightarrow\mathcal M_{m+2,\ell}$.
For each parity the dimensions are nondecreasing and bounded, and
therefore eventually constant. These injections are consequently
isomorphisms for all sufficiently large $m$ of either parity.
Choose a vector-space complement to
$\F\mathcal M_{m-2,\ell}$ in $\mathcal M_{m,\ell}$ at every degree,
with the source interpreted as zero for $m<2$.
Only finitely many complements are nonzero. Bases of these complements
and their successive multiples by $\F$ give a homogeneous free basis
of $\mathcal M_\ell$. This proves both finite generation and freeness.
\end{proof}

\begin{remark}\label{rem:eventual-divisibility}
There exists $m_0(\ell)$ such that every
$G\in\mathcal M_{m,\ell}$ with $m\ge m_0(\ell)$ is divisible by
$\F$, with quotient in $\mathcal M_{m-2,\ell}$.
This assertion about the entire Siegel modular form follows from
boundedness and injectivity in Theorem~\ref{thm:finite-module}.
It is stronger than the assertion about the leading coefficient in
Proposition~\ref{prop:bounded-leading}; the latter alone would not
justify division of the entire modular form.
\end{remark}

\subsection{Numerical dimension zero}
For a Cartier divisor $D$ on a smooth projective variety $X$ and an
ample Cartier divisor $A$, define
\[
 \sigma(D;A)=
 \max\left\{q\in\Z_{\ge0}:
       \limsup_{m\to\infty}
       \frac{h^0(X,mD+A)}{m^q}>0\right\},
\]
with value $-\infty$ if the set is empty. The numerical Kodaira
dimension used here is
\[
 \nu_\sigma(D)=\max_{A\ \mathrm{ample}}\sigma(D;A).
\]
This is Nakayama's section-growth definition
\cite[Chapter~V, Section~2]{Nakayama04}.
For the case $\nu_\sigma(D)=0$, the standard versions of numerical
dimension agree; see \cite[Definition~2.5 and Remark~2.8]{Gongyo11}.

We recall the terminology for the birational assertions below.
A normal variety $Y$ is $\Q$-factorial if every Weil divisor has a
positive Cartier multiple. If $K_Y$ is $\Q$-Cartier and
$f:W\to Y$ is a resolution, write
\[
 K_W=f^*K_Y+\sum_E a(E,Y)E,
\]
where the sum runs over the exceptional prime divisors on $W$.
The number $a(E,Y)$ is the discrepancy of the associated divisorial
valuation. Throughout, divisorial valuations are normalized to have
value group $\Z$. If $v=\ord_F$, we write $a(v,Y)=a(F,Y)$;
for a $\Q$-Cartier divisor $D$, the value $v(D)$ is the coefficient
of $F$ in its pullback to a model carrying $F$.
The singularities of $Y$ are terminal when every
exceptional divisorial valuation has positive discrepancy.
A terminal minimal model has nef canonical divisor; it is a good
minimal model when a positive multiple of the canonical divisor is
generated by global sections. We use the existence and abundance
theorems in numerical dimension zero, not an abundance assertion for
arbitrary Kodaira dimension.

\begin{theorem}\label{thm:numerical-zero}
For a smooth projective model $X$ of $\A_6$ one has
\[
 \nu_\sigma(K_X)=0.
\]
The variety $\A_6$ has a projective $\Q$-factorial terminal good
minimal model $Y$ satisfying
\begin{equation}\label{eq:canonical-torsion}
 2K_Y\sim0,\qquad K_Y\not\sim0.
\end{equation}
\end{theorem}
\begin{proof}
First choose $X$ to resolve $\A_6^{\mathrm{Perf}}$, and let $L_X$
be the big Cartier Hodge multiple fixed in
\eqref{eq:canonical-injection}. For every fixed $a\ge0$, that
injection and Theorem~\ref{thm:finite-module} give
\[
 \sup_{m\ge0}h^0(X,mK_X+aL_X)<\infty.
\]
Let $A$ be any ample Cartier divisor. Bigness of $L_X$ implies that
$aL_X-A$ has a nonzero section for some $a>0$.
Multiplication by this fixed section gives, for all $m$,
\[
 H^0(X,mK_X+A)\lhook\joinrel\longrightarrow
                      H^0(X,mK_X+aL_X).
\]
Thus $\sigma(K_X;A)\le0$ for every ample $A$. The effective
bicanonical divisor shows that $K_X$ is pseudo-effective.
Equivalently, for a sufficiently positive $A$, even multiples of
the known bicanonical section already give nonzero sections of
$mK_X+A$. Hence $\nu_\sigma(K_X)=0$.
Numerical Kodaira dimension of the canonical divisor is unchanged
by smooth birational modifications: pullback and addition of an
effective exceptional divisor do not change it
\cite[Lemma~2.6(1)--(2)]{Gongyo11}. This proves the assertion for
every smooth projective model.

Choose an ample $\Q$-divisor $H$ such that $(X,H)$ is klt and
$K_X+H$ is nef. The minimal-model theorem for numerical dimension zero
\cite[Theorem~1.2 and Theorem~5.1]{Gongyo11} gives a terminating
$K_X$-minimal model program with scaling of $H$.
Starting with a smooth variety yields a $\Q$-factorial terminal
minimal model $Y$. Since $K_Y$ is nef and has numerical dimension
zero, it is numerically trivial \cite[Lemma~2.6(3)]{Gongyo11}.
Abundance for this case \cite[Lemma~6.2]{Gongyo11} gives
$K_Y\sim_\Q0$.

The canonical algebra is birationally invariant for terminal varieties.
Thus Corollary~\ref{cor:canonical-ring} gives a nonzero section of
the reflexive sheaf $\OO_Y(2K_Y)$ and no nonzero section of
$\OO_Y(K_Y)$. The effective zero divisor of the former is numerically
trivial. Its intersection with an ample divisor to the power
$\dim Y-1$ is zero, so the divisor itself is zero. It follows that
$2K_Y\sim0$ as a Weil divisor, and in particular $2K_Y$ is Cartier.
The vanishing of $H^0(Y,\OO_Y(K_Y))$ gives $K_Y\not\sim0$.
\end{proof}

\subsection{The exceptional divisor and the canonical double cover}
Let $E_{\F}$ denote the closure of the zero divisor of $\F$ in
$X_0=\A_6^{\mathrm{Perf}}$. It is prime and reduced by
Corollary~\ref{cor:F-prime}. The canonical divisor formula and the
boundary order of $\F$ give
\begin{equation}\label{eq:canonical-prime-divisor}
 2K_{X_0}\sim E_{\F},\qquad
 [E_{\F}]=14\lambda_6-2[D_6].
\end{equation}
Indeed the order two of $\F$ cancels the double pole of a
bicanonical volume form along the rank-one boundary.

\begin{theorem}\label{thm:minimal-geometry}
The minimal model in Theorem~\ref{thm:numerical-zero} can be chosen
with Picard number one. For every terminal good minimal model $Y$
of $\A_6$, the divisorial valuation defined by $E_{\F}$ is exceptional
over $Y$ and satisfies
\begin{equation}\label{eq:half-discrepancy}
 a(E_{\F},Y)=\frac12.
\end{equation}
The divisor $E_{\F}$ is uniruled. The canonical divisor of $Y$ is
not Cartier along the center of this valuation, and its Cartier
index there is two. In particular, $\A_6$ has no smooth projective
minimal model.
\end{theorem}
\begin{proof}
Let $W$ be a common smooth resolution of $X_0$ and $Y$.
The unique bicanonical section on $W$ has order one along the strict
transform of $E_{\F}$: its generic point is a smooth interior
point of $X_0$, where $\F$ has a simple zero.
On $Y$, the corresponding section of $2K_Y$ has no zeros by
\eqref{eq:canonical-torsion}. If $E_{\F}$ were not exceptional
over $Y$, the orders at its generic point would coincide, which is
impossible. Comparing the two sections under the canonical divisor
formula on $W$ gives
$2a(E_{\F},Y)=1$. This proves \eqref{eq:half-discrepancy}.

For the Picard-number assertion, use that $X_0$ is $\Q$-factorial
with Picard number two \cite[Theorem~1.4]{SB06}, and has terminal
singularities \cite[Theorem~1.1]{SB16}. The $\Q$-factoriality here
is a statement about the full-level coarse variety. It can also be
seen directly: Freitag's theorem makes every interior prime divisor
the exact zero divisor of a modular form, so its closure has class
$k\lambda_6-bD_6$ in the rational divisor class group. Both
$\lambda_6$ and $D_6$ are $\Q$-Cartier, and $D_6$ is the only
boundary prime. Hence every Weil divisor is $\Q$-Cartier.
Numerical dimension zero
descends from a resolution by \cite[Lemma~2.6(1)--(2)]{Gongyo11},
so run the terminating $K_{X_0}$-minimal model program with scaling.
At every stage before $E_{\F}$ is contracted, the transform of
\eqref{eq:canonical-prime-divisor} remains valid.
Suppose that a divisorial contraction had a different exceptional
prime divisor $E'$. Its contracted curves cover $E'$, and a general
such curve $C$ is not contained in $E_{\F}$.
Effectivity and $\Q$-Cartierness would imply $E_{\F}\cdot C\ge0$,
whereas $E_{\F}\cdot C=2K\cdot C<0$. This is a contradiction.
Thus the only possible divisorial contraction removes $E_{\F}$.
It must occur, by \eqref{eq:half-discrepancy}.
Flips preserve Picard number, and this divisorial contraction lowers
it from two to one. After it, the pushforward canonical divisor
satisfies $2K\sim0$, so the program has reached a good minimal model.

The exceptional divisor of a $K$-negative divisorial contraction is
covered by rational curves; this follows from the theorem on the
length of extremal rational curves \cite[Theorem~1]{Kawamata91},
which covers each positive-dimensional fiber component by rational
curves.
Uniruledness is a birational property, so $E_{\F}$ is uniruled.

Finally, if $K_Y$ were Cartier near the center of $E_{\F}$, the
coefficient of that exceptional divisor in
$K_W-f^*K_Y$ would be an integer. This contradicts
\eqref{eq:half-discrepancy}. As $2K_Y$ is Cartier, the local
Cartier index is precisely two. If a smooth projective minimal model
existed, its nef canonical divisor would have numerical dimension
zero by Theorem~\ref{thm:numerical-zero}. It would therefore be
numerically trivial, and then $\Q$-linearly trivial by
\cite[Lemmas~2.6(3) and~6.2]{Gongyo11}. Thus it would be a good
minimal model to which the same discrepancy argument applies.
Its canonical divisor would be Cartier everywhere, a contradiction.
\end{proof}

\begin{corollary}\label{cor:canonical-cover}
There is a connected projective double cover
$\pi:\widetilde Y\to Y$, \'etale in codimension one, such that
$\widetilde Y$ has terminal singularities and $K_{\widetilde Y}\sim0$.
\end{corollary}
\begin{proof}
Choose the trivialization of $\OO_Y(-2K_Y)$ supplied by
\eqref{eq:canonical-torsion}, and form the normalization of
\[
 \Spec_Y\bigl(\OO_Y\oplus\OO_Y(-K_Y)\bigr),
\]
using reflexive multiplication. Since $K_Y$ has exact order two in the divisor class group, this
cover is connected. Indeed, if the associated quadratic algebra
split over $\C(Y)$, a rational function trivializing $2K_Y$
would be a square, making $K_Y$ principal. Its square root on
the cover trivializes the pullback of $K_Y$. The cover is
\'etale at every codimension-one point, and the canonical
bundle formula therefore gives
$K_{\widetilde Y}=\pi^*K_Y\sim0$.
Terminality is preserved by this finite cover: the discrepancy
formula \cite[Proposition~5.20]{KM98} is
\[
 a(F',\widetilde Y)+1
      =e(F'/F)\bigl(a(F,Y)+1\bigr),
\]
where $e(F'/F)\ge1$ is the divisorial ramification index.
An exceptional valuation upstairs maps to an exceptional valuation
downstairs, so positivity of the discrepancies is preserved.
\end{proof}

\begin{remark}
The double cover is not asserted to be smooth, and no vanishing
$H^i(\widetilde Y,\OO_{\widetilde Y})=0$ for intermediate $i$ is used
or proved. In particular the conclusion is stated in terms of
terminal singularities and a trivial canonical divisor, without
imposing an additional definition of a Calabi--Yau variety.
Likewise, \eqref{eq:half-discrepancy} does not specify the dimension
or local analytic type of the center of $E_{\F}$.
\end{remark}

\begin{remark}\label{rem:modular-canonical-cover}
The canonical cover has a direct modular description. On the smooth
interior, let $s$ be the bicanonical section defined by $\F$ and
choose a rational top differential $\eta$. Writing $s=f\eta^2$
defines the quadratic function-field extension $\C(\A_6)(\sqrt f)$,
independently of the choice of $\eta$. This is the function field of
$\widetilde Y$. In modular coordinates it is given by
\[
 w^2=\F(\Omega),\qquad
 (\Omega,w)\longmapsto
 \bigl(\gamma\Omega,\det(C\Omega+D)^7w\bigr).
\]
The normalization of this modular model is branched along the
reduced interior divisor $E_{\F}$.
Theorem~\ref{thm:prym-transverse-A1}, proved below, shows that along a general
Prym point, it has a three-dimensional ordinary double point in
the transverse directions. The valuation $E_{\F}$ has a unique
extension $R$ to this quadratic field, with ramification index two.
The discrepancy formula gives
\[
 a(R,\widetilde Y)+1=2\bigl(a(E_{\F},Y)+1\bigr)=3.
\]
This ramification is compatible with the absence of codimension-one
ramification over $Y$, since $E_{\F}$ is exceptional over $Y$.
\end{remark}

\section{Transverse singularities along the Prym locus}
\label{sec:singularities}

We retain the notation $\phi=\phi_2(\F)$ and
$\Sigma_1,\ldots,\Sigma_{27}$ from Sections~\ref{sec:prym}
and~\ref{sec:theta}. We shall determine the transverse singularity
of $\divi(\phi)$ along these surfaces. Only its local equation near
the origin is used; global integrality of the divisor is not required.

For a sufficiently general $(A,\Theta)$, let $f$ be a local equation
of $\divi(\phi)$ at $0$, and put
\[
 T_{10}=\operatorname{in}_0(f),\qquad
 X=V(T_{10})\subset\mathbf P T_0A=\mathbf P^4.
\]
The proof of Proposition~\ref{prop:theta-normal} gives
\[
 \deg X=10,\qquad X\text{ integral},\qquad \dim\Sing X\le1.
\]
We identify each curve $C_i$ with its Prym-canonical image in
$\mathbf P T_0A$. By the proof of
Proposition~\ref{prop:prym-boundary-surfaces}, it is the support of
$\mathbf P C_0\Sigma_i$, where $C_0$ denotes the tangent cone.
These curves are nondegenerate, have degree $10$, and are pairwise
distinct: the embedded curve recovers
$\eta_i=\OO_{C_i}(1)\otimes \omega_{C_i}^{-1}$ and hence its Prym presentation.
Corollary~\ref{cor:prym-leading-order} gives
$\ord_{\Sigma_i}(\phi)\ge2$. Thus $df$ vanishes on the reduced
surface $\Sigma_i$, and its initial terms show that
\begin{equation}\label{eq:prym-tangent-singular}
 C_i\subset\Sing X\qquad(1\le i\le27).
\end{equation}

\begin{lemma}\label{lem:gauss-image-dimension}
Let $X\subset\mathbf P^4$ be an integral hypersurface of degree
$d\ge2$ with $\dim\Sing X\le1$. If $\Sing X$ contains an integral curve
which is not a line, then the image of the Gauss map
\[
 \gamma:X_{\mathrm{sm}}\longrightarrow(\mathbf P^4)^*,
 \qquad x\longmapsto T_xX,
\]
has dimension at least two.
\end{lemma}

\begin{proof}
Write $r=\dim\overline{\gamma(X_{\mathrm{sm}})}$. The closure of a
general Gauss fiber is a linear space of dimension $3-r$
\cite[(2.10), pp.~388--389]{GH79}. The case $r=0$ would make $X$
a hyperplane. Suppose $r=1$, and denote the planes arising as
closures of general fibers by $\Pi$.

If $T$ defines $X$, the restrictions of its partial derivatives to
$\Pi$ have constant projective ratios. They are therefore constant
multiples of one nonzero homogeneous polynomial $P_\Pi$ of degree
$d-1$. The plane curve $V(P_\Pi)$ is contained in $\Sing X$.
There are finitely many curve components of $\Sing X$. Since the
planes form an irreducible one-dimensional family and containment
of a fixed curve is a closed condition, one such component is
contained in every general $\Pi$. A curve spanning a plane would
force $\Pi$ to be fixed. Consequently this component is a line
$\ell$, and $X$ is the cone with vertex $\ell$ over an integral
plane curve $B$ of degree $d$.

If $B$ were singular, the plane joining $\ell$ to a singular point
of $B$ would lie in $\Sing X$, contrary to its dimension bound.
Thus $B$ is smooth. The equation of the cone depends only on the
three coordinates transverse to $\ell$, so $\Sing X=\ell$.
This contradicts the hypothesis on the singular curve.
\end{proof}

\begin{lemma}\label{lem:tangent-polar-cycle}
Let $X=V(T)\subset\mathbf P^4$ be an integral hypersurface of
degree $d\ge2$ with $\dim\Sing X\le1$, and let
$B_1,\ldots,B_s$ be distinct integral curves in $\Sing X$.
Two homogeneous polars with general directions $v,w\in\C^5$
define a complete-intersection curve
\[
 Q=V(T,\partial_vT,\partial_wT)
\]
of degree $d(d-1)^2$. If $m_i$ is the coefficient of $B_i$ in
the fundamental cycle of $Q$, then $m_i\ge2$. If the quadratic
normal term of $T$ at the generic point of $B_i$ has rank at most
two, then $m_i\ge3$. Moreover, if the Gauss image of $X$ has
dimension at least two, then
\begin{equation}\label{eq:tangent-polar-degree}
 \sum_{i=1}^s m_i\deg B_i<d(d-1)^2.
\end{equation}
\end{lemma}

\begin{proof}
The polar system on $X$ has base locus $\Sing X$. A general
polar cuts a surface on $X$, and no component of this surface is
contained in the base locus. A second general polar avoids these
components. Thus $T,\partial_vT,\partial_wT$ is a regular
sequence and $\deg Q=d(d-1)^2$.

At the generic point of $B_i$, the ambient local ring is regular
of dimension three. Its completion is $K[[x_1,x_2,x_3]]$, where
$K$ is the function field of $B_i$, and a local equation of $X$ is
\[
 f=q_2(x_1,x_2,x_3)+O((x_1,x_2,x_3)^3).
\]
The two local polar equations belong to the maximal ideal.
Lemma~\ref{lem:local-length} therefore gives $m_i\ge2$.
Their linear terms are contractions of $q_2$. Indeed, after
dehomogenization, a homogeneous polar is a local directional
derivative of $f$ plus a multiple of $f$; differentiation along
$B_i$ contributes only terms of order at least two.

Suppose that $\rk q_2\le2$. If the linear terms of the two polar
equations are dependent, their finite local quotient together
with $f$ has embedding dimension at least two and hence length
at least three. Otherwise $\rk q_2=2$, and the common kernel of
these linear terms is the radical of $q_2$. Eliminating two
variables by the formal implicit function theorem leaves an
equation of order at least three in the remaining variable.
Again $m_i\ge3$.

It remains to prove strictness under the additional hypothesis.
Let $Y$ be the closure of the Gauss image and put $r=\dim Y\ge2$.
A general codimension-two linear section of $Y$ meets the open
set over which the Gauss map has fibers of dimension $3-r$.
Its inverse image in $X_{\mathrm{sm}}$ is nonempty and has
dimension $(r-2)+(3-r)=1$. Its closure is a component of $Q$
meeting $X_{\mathrm{sm}}$. This component has positive degree
and differs from every $B_i$, which proves
\eqref{eq:tangent-polar-degree}.
\end{proof}

\begin{theorem}\label{thm:prym-transverse-A1}
For a sufficiently general principally polarized abelian fivefold $(A,\Theta)$,
the degree-ten hypersurface $V(T_{10})\subset\mathbf P T_0A$
has a transverse $A_1$ singularity at the general point of each
$C_i$, and $\divi(\phi)$ has a transverse $A_1$ singularity
at the general point of each $\Sigma_i$. The divisor $E_{\F}$
likewise has a transverse $A_1$ singularity at the general point
of $\mathcal P_7$. Their local analytic equations at these points
can be written
\[
 x_1^2+x_2^2+x_3^2=0,
\]
where $x_1,x_2,x_3$ are coordinates normal to the respective loci.
\end{theorem}

\begin{proof}
Apply Lemma~\ref{lem:gauss-image-dimension} to $X=V(T_{10})$,
using \eqref{eq:prym-tangent-singular}. If the quadratic normal
terms of $T_{10}$ were degenerate at the general point of every
$C_i$, Lemma~\ref{lem:tangent-polar-cycle} would give
\[
 810=27\cdot10\cdot3
 \le\sum_{i=1}^{27}m_i\deg C_i
 <10\cdot9^2=810,
\]
a contradiction. Thus the quadratic normal term is nondegenerate
at the general point of at least one $C_i$.

We now use the irreducibility of $\mathcal R_6$ to obtain this
assertion for all $27$ curves on a sufficiently general fivefold.
Over a suitable open subset of $\mathcal R_6$, the universal
curve has its relative Prym-canonical embedding in the projective
bundle of the tangent spaces of the Prym varieties. The quadratic
normal jet of $T_{10}$ is algebraic, and its determinant is
nonzero on an open subset of this universal curve. The preceding
argument shows that this open subset is nonempty. The universal
curve is irreducible, so, after shrinking $\mathcal R_6$, the
nonvanishing locus meets each curve fiber in a dense open subset.
Restrict the degree-$27$ Prym map to an open subset of
$\mathcal A_5$ over which it is finite and \'etale. Removing
the image of the closed complement of the chosen open subset
of $\mathcal R_6$ ensures that all $27$ Prym presentations have
the asserted property. The holomorphic Morse lemma with parameters
proves the assertion about $V(T_{10})$.

Fix $i$ and choose a nonzero vector $p\in C_0\Sigma_i$ over a
general point of $C_i$. The affine Hessian of $T_{10}$ at $p$
has rank three: the radial direction is in its kernel, and the
quadratic normal form in projective space is nondegenerate.
We transfer this rank statement to the local equation $f$ of
$\divi(\phi)$ without any multiplicity assumption on
$C_0\Sigma_i$.

Blow up $\Sigma_i$ at the origin. The support of its exceptional
fiber is $\Supp(\mathbf P C_0\Sigma_i)=C_i$. An analytic curve through
the point $[p]$ of this exceptional fiber, with general point
over $(\Sigma_i)_{\mathrm{sm}}\setminus\{0\}$, gives after
normalization an arc
\[
 q(s)=s^m\bigl(p+O(s)\bigr)\in\Sigma_i,
 \qquad m\ge1.
\]
Here a nonzero constant rescaling of $p$ has been absorbed into
the parameter. Since $f=T_{10}+O((z)^{11})$, one has
\begin{equation}\label{eq:tangent-hessian-arc}
 \operatorname{Hess}f(q(s))
 =s^{8m}\bigl(\operatorname{Hess}T_{10}(p)+O(s)\bigr).
\end{equation}
A nonzero $3\times3$ minor on the right remains nonzero for
general $s\ne0$. Thus the Hessian of $f$ has rank at least
three there. On the other hand, $df$ vanishes identically on
the reduced surface $\Sigma_i$. Differentiating this identity
at its smooth points shows that its two-dimensional tangent
space lies in the kernel of the Hessian. The rank is therefore
exactly three. Nondegeneracy is an open condition, so it holds
at the general point of $\Sigma_i$. The same Morse lemma
proves the assertion about $\divi(\phi)$.

Finally, Theorem~\ref{thm:prym-order} gives
$\ord_{\mathcal P_7}(\F)\ge2$. At a general point of
$\Delta'_0$, the proof of
Corollary~\ref{cor:prym-leading-order} identifies the immersed
Prym branch with $(y_1,y_2,y_3)=0$ and the boundary parameter
with $q=t$ up to a unit. The normal directions of this branch
restrict to the normal directions of $\Sigma_i$ in the abelian
fiber. The holomorphic section $\F/q^2$ belongs to
$(y_1,y_2,y_3)^2$: the parameter $q$ is a nonzerodivisor modulo
this ideal, so membership away from $q=0$ extends across the
boundary. Its quadratic normal jet restricts to that of $\phi$.
Its determinant is nonzero at a general point of the boundary
surface and hence is not identically zero along $\mathcal P_7$.
The Morse lemma proves the remaining assertion.
\end{proof}

\begin{corollary}\label{cor:prym-crepant-blowup}
Over suitable dense open subsets of each $\Sigma_i$ and of
$\mathcal P_7$, blowing up the ambient smooth variety along the
Prym locus gives a crepant resolution of the divisor near the center.
Its exceptional divisor is a smooth conic bundle
over that locus. Locally, the double cover of the ambient variety
branched along the divisor has a three-dimensional ordinary double
point transverse to the Prym locus.
\end{corollary}

\begin{proof}
The normal quadratic form in Theorem~\ref{thm:prym-transverse-A1}
is nondegenerate. Its zero locus in the projectivized normal
bundle is therefore a smooth conic bundle, and the strict transform
of the divisor is smooth near this exceptional divisor. An ambient
blowup in codimension three has discrepancy two, while the divisor
has multiplicity two along its center. Adjunction gives discrepancy
$2-2=0$ on its strict transform. Finally, the branched double cover
has local equation
\[
 u^2=x_1^2+x_2^2+x_3^2,
\]
a nondegenerate quadratic hypersurface in four transverse variables.
\end{proof}

\section{The moving slope}
\label{sec:moving}

An effective divisor is moving if its complete linear system has no
fixed prime component and has positive dimension. For a rational
divisor we use a sufficiently divisible integral multiple. The moving slope
$s_{\mathrm{Mov}}(\A_6)$ is the infimum of the slopes of such divisors
on $\A_6^{\mathrm{Perf}}$. We prove the upper bound proposed in
\cite[Corollary~E]{GIMS24} by combining the transverse quadratic
singularity along the Prym locus with the geometry of Nikulin surfaces.

\begin{theorem}\label{thm:moving-form}
There is a scalar Siegel modular form $G_{86}\in\MM_{86}^{(6)}$ such that
\[
 b(G_{86})=12,\qquad \F\nmid G_{86}.
\]
Consequently
\[
 s_{\mathrm{Mov}}(\A_6)\le\frac{43}{6}.
\]
\end{theorem}

\subsection{The discriminant of a Nikulin net of quadrics}
Retain the standard Nikulin surface $S$ and the classes $H,e,L=H-e$
from Section~\ref{sec:prym}. By
\eqref{eq:nikulin-section-dimensions}, the embedding
$S\subset\mathbf P^5$ is the complete intersection of the net
\[
 W=I_2(S),\qquad \dim W=3.
\]
Choose a basis $q_0,q_1,q_2$ of $W$ and write
$q_j(z)=z^tA_jz$, where the $A_j$ are symmetric $6\times6$ matrices.
The discriminant of the net is
\[
 \Delta_S=V\!\left(\det(s_0A_0+s_1A_1+s_2A_2)\right)
                    \subset\mathbf P(W).
\]

\begin{lemma}\label{lem:nikulin-discriminant}
If $\Pic(S)=\Z H\mathbin{\perp}\mathfrak N$, then
$\Delta_S$ is a smooth irreducible plane sextic. In particular,
this holds for a general standard Nikulin surface of genus seven.
\end{lemma}
\begin{proof}
We first show that no class $M\in\Pic(S)$ satisfies
$M^2=0$ and $L\cdot M=4$. Write
\[
 M=aH+\sum_{i=1}^8 b_iN_i,
 \qquad a\in\Z,\quad b_i\in\tfrac12\Z.
\]
The additional parity conditions defining $\mathfrak N$ are not
needed. Since $H^2=12$, $H\cdot N_i=0$ and $L\cdot N_i=1$,
the two assumed equalities give
\[
 \sum_{i=1}^8b_i^2=6a^2,\qquad
 \sum_{i=1}^8b_i=4-12a.
\]
The Cauchy--Schwarz inequality implies
\[
 (4-12a)^2\le48a^2,
 \qquad\text{hence}\qquad 6a(a-1)+1\le0,
\]
which is impossible for an integer $a$.

Every quadric in the net is nonsingular at every point of $S$:
otherwise the Jacobian of its three defining equations would have
rank at most two there. A quadric of rank at most three has a
vertex of dimension at least two. The other two quadrics necessarily
have a common zero on that vertex, contradicting the smoothness
of $S$. A quadric of rank four has a vertex line disjoint from $S$.
Projection from this line restricts to a morphism
\[
 \pi:S\longrightarrow\mathbf P^1\times\mathbf P^1,
 \qquad \pi^*\OO(1,1)=L.
\]
The ampleness of $L$ makes $\pi$ finite and surjective. Since
$L^2=8$, its degree is four. The class $M=\pi^*\OO(1,0)$ then
satisfies $M^2=0$ and $L\cdot M=4$, again a contradiction.
Thus every singular quadric in the net has rank five.

Put $A(s)=\sum_js_jA_j$. At a rank-five matrix $A(s)$, let $v$
span its kernel. For some $c\ne0$ one has
\[
 \operatorname{adj}(A(s))=c\,vv^t,\qquad
 \frac{\partial\det A(s)}{\partial s_j}
    =\tr\!\left(\operatorname{adj}(A(s))A_j\right)=c\,q_j(v).
\]
If all three derivatives vanished, the point $[v]$ would belong
to $S$ and to the vertex of the corresponding quadric, which is
impossible. This also shows that $\det A(s)$ cannot vanish
identically: otherwise every member would have rank five and all
three derivatives would vanish. Hence $\Delta_S$ is a smooth plane
sextic. Two positive-degree plane curves intersect, so smoothness
also implies irreducibility. Finally, smoothness of the discriminant
is open in families, giving the assertion for a general surface.
\end{proof}

\subsection{The gradient along the distinguished divisor}
For a symmetric period matrix $\Omega$, use the normalization
\[
 \partial_{ij}=\frac{1+\delta_{ij}}2
                  \frac{\partial}{\partial\Omega_{ij}},
 \qquad \partial=(\partial_{ij})_{1\le i,j\le6}.
\]
For a local holomorphic function $f$ on a smooth Siegel chart,
the matrix $\partial f$ represents its differential as a quadratic
form under the local Kodaira--Spencer identification of the
cotangent bundle with $\Sym^2\mathbb E_6$.
In particular the off-diagonal entries contain the factor $1/2$.
Condition~$(\star)$ of \cite{GIMS24} requires that
$\det(\partial f)$ not vanish identically on any component of the
zero divisor of the modular form. Along that divisor, replacing
$f$ by $uf$ for a holomorphic unit $u$ multiplies its differential
by $u$, so this condition does not depend on the trivialization.

\begin{proposition}\label{prop:F-gradient-nondegenerate}
The form $\F$ satisfies Condition~$(\star)$ of \cite{GIMS24}.
\end{proposition}
\begin{proof}
Choose a general point $p=P_7(C,\eta)$ of $\mathcal P_7$ and work
in a smooth local Siegel chart. The conormal identification
\eqref{eq:nikulin-multiplication} gives
\begin{equation}\label{eq:prym-conormal-net}
 N_p^*:=N^*_{\mathcal P_7/\A_6,p}
   =I_2(C,\omega_C\otimes\eta)=I_2(S)=W,
\end{equation}
where $S$ is the general standard Nikulin surface containing
$(C,\eta)$. These are subspaces of
$\Sym^2H^0(C,\omega_C\otimes\eta)$, and their elements are
quadrics on $\mathbf P(H^0(C,\omega_C\otimes\eta)^*)$.

Let $f$ be a local equation of the divisor of $\F$.
By Theorem~\ref{thm:prym-transverse-A1}, its quadratic normal term
$q\in\Sym^2N_p^*$ is nondegenerate. Differentiating this quadratic
form defines an isomorphism
\[
 \beta_q:N_p\longrightarrow N_p^*,\qquad v\longmapsto dq_v.
\]
The image under $\mathbf P(\beta_q)$ of the isotropic conic
$V(q)\subset\mathbf P(N_p)$ is therefore a smooth conic
$\Gamma_p\subset\mathbf P(W)$. By
Lemma~\ref{lem:nikulin-discriminant}, it cannot be contained in
$\Delta_S$. Choose $v\ne0$ with
\[
 q(v)=0,\qquad \det\bigl(\beta_q(v)\bigr)\ne0,
\]
where the second determinant is that of the quadric matrix in
\eqref{eq:prym-conormal-net}.

Take local coordinates $(x,y)$ with
$\mathcal P_7=\{x_1=x_2=x_3=0\}$ and $p=(0,0)$.
Then $f\in(x_1,x_2,x_3)^2$ and
$f(x,0)=q(x)+O(|x|^3)$. The holomorphic Morse lemma on the
transverse slice $y=0$ gives an analytic arc $\gamma(t)$ in
$V(f)$ whose first normal direction is $v$. In the original
period coordinates its differential has expansion
\[
 df\bigl(\gamma(t)\bigr)=t\,\beta_q(v)+O(t^2).
\]
The tangential components have order at least two because
$f\in(x_1,x_2,x_3)^2$. It follows that
\[
 \det(\partial f)\bigl(\gamma(t)\bigr)
   =t^6\det\bigl(\beta_q(v)\bigr)+O(t^7)\ne0
\]
for sufficiently small nonzero $t$. Thus the determinant does not
vanish identically on the divisor of $\F$. That divisor is prime
and reduced by Corollary~\ref{cor:F-prime}, so this proves
Condition~$(\star)$.
\end{proof}

\begin{proof}[Proof of Theorem~\ref{thm:moving-form}]
By \cite[Theorem~A]{GIMS24}, the differential operator
$\mathcal D_{6,14}$ applied to $\F$ gives a scalar form
\[
 G_{86}:=\mathcal D_{6,14}(\F)\in\MM_{86}^{(6)},
 \qquad b(G_{86})\ge6b(\F)=12.
\]
Its restriction to $V(\F)$ is a nonzero constant multiple of
$\det(\partial\F)$.
Proposition~\ref{prop:F-gradient-nondegenerate} therefore implies
$\F\nmid G_{86}$ and, in particular, $G_{86}\ne0$.
Theorem~\ref{thm:minimal-slope} now gives
\[
 12\le b(G_{86})\le\frac{86}{7}<13.
\]
Since the boundary order is an integer, it is exactly twelve.

The two forms $\F^{43}$ and $G_{86}^{7}$ have weight $602$
and boundary orders $86$ and $84$, respectively. As sections of
$602\lambda_6-84[D_6]$, they have no common prime component:
the only possible interior component is $V(\F)$, and the second
section does not vanish along the boundary. They therefore define
a moving pencil of slope $602/84=43/6$.
\end{proof}

\section{Divisor cones and the boundary of the minimal model}
\label{sec:birational}

In this section put $X_0=\A_6^{\mathrm{Perf}}$, $\lambda=\lambda_6$,
$D=D_6$ and $E=E_{\F}$. The minimal model program constructed in
Theorem~\ref{thm:minimal-geometry} has the form
\begin{equation}\label{eq:finite-MMP}
 X_0\dashrightarrow X_1\dashrightarrow\cdots\dashrightarrow X_r
       \xrightarrow{\ f\ }Y,
\end{equation}
where the dashed arrows are flips and $f$ contracts precisely the
strict transform of $E$. The models $X_i$ have Picard number two,
and $Y$ has Picard number one. A subscript $i$ denotes strict
transform on $X_i$. We identify the divisor spaces of the $X_i$
using these isomorphisms in codimension one.

\subsection{The effective and movable cones}
\begin{theorem}\label{thm:effective-cone}
The pseudo-effective cone is
\[
 \overline{\Eff}(X_0)
   =\mathbb R_{\ge0}[D]+\mathbb R_{\ge0}(7\lambda-[D]).
\]
A real class $a\lambda-b[D]$ is big if and only if
$a>0$ and $a>7b$. Every nonzero effective rational divisor on the
second extremal ray is a positive rational multiple of $E$.
\end{theorem}
\begin{proof}
Freitag's exact-divisor theorem realizes an interior prime divisor
as the zero divisor of a scalar form of weight $k>0$ and boundary
order $b\ge0$. The class of its closure is
\[
 k\lambda-b[D]
 =\frac{k}{7}(7\lambda-[D])+
                         \left(\frac{k}{7}-b\right)[D].
\]
Theorem~\ref{thm:minimal-slope} makes both coefficients nonnegative.
The only remaining prime boundary divisor is $D$. Conversely,
$D$ and $E\sim2(7\lambda-D)$ are effective. This proves the
cone assertion, and its interior is the big cone. On the second
extremal ray, each prime summand must have slope seven. The equality
case of Theorem~\ref{thm:minimal-slope} identifies it with $E$.
\end{proof}

We use the definition of a Mori dream space from
\cite[Definition~1.10]{HK00}: a normal projective $\Q$-factorial
variety with $\Pic_\Q=N^1_\Q$, whose nef cone and the nef cones
of finitely many small $\Q$-factorial modifications are generated
by semiample classes, and whose movable cone is the union of
these nef cones.

\begin{theorem}\label{thm:mori-dream}
The perfect cone compactification $X_0$ is a Mori dream space.
There is a rational number $\sigma$ such that
\begin{equation}\label{eq:movable-cone}
 \overline{\operatorname{Mov}}(X_0)
   =\mathbb R_{\ge0}\lambda+
                  \mathbb R_{\ge0}(\sigma\lambda-[D]),
 \qquad 7<\sigma\le\frac{43}{6}.
\end{equation}
In particular, $s_{\mathrm{Mov}}(\A_6)=\sigma$ is attained.
The class $\sigma\lambda_r-D_r$ is semiample and defines the
last contraction $f$ in \eqref{eq:finite-MMP}.
\end{theorem}
\begin{proof}
The $\Q$-factoriality and the equality $\Pic_\Q=N^1_\Q$ follow
from \cite[Theorem~1.4]{SB06}; they are preserved by the small
modifications in \eqref{eq:finite-MMP}. We prove that both nef
boundary rays on every $X_i$ are semiample.

On $X_0$, one ray is the Hodge ray $\lambda$, which defines the
Satake morphism. Its contracted curves cover $D$; a general such
curve $C$ is not contained in $E$, so $K_{X_0}\cdot C=E\cdot C/2\ge0$.
The canonical divisor is not nef: otherwise $\nu_\sigma(K_{X_0})=0$
would make it numerically trivial, contrary to the nonzero
effective divisor $E\sim2K_{X_0}$. Since the cone of curves is
two-dimensional, its other boundary ray is $K$-negative.
The contraction theorem gives a contraction of this ray, and
the pullback of an ample rational class on its target generates
the other nef boundary. It is semiample.

After any flip, the incoming contracted curve ray is $K$-positive.
The same argument applies: $2K_{X_i}\sim E_i\ne0$ and
$\nu_\sigma(K_{X_i})=0$, so the other curve ray is $K$-negative. The
incoming and outgoing contractions therefore give the two
semiample nef boundaries on $X_i$. On $X_r$, the outgoing one
is $L_*=f^*A$, where $A$ is an ample rational divisor on $Y$.
The nef cones are rational polyhedral, and their adjacent members
share the ray pulled back from the target of the corresponding flip.

It remains to prove that these finitely many nef cones exhaust
the movable cone. A movable divisor has nonnegative intersection
with a general curve in a family covering a divisor, on each $X_i$.
Indeed one can choose an effective member without that prime
component, and a general curve of the family is not contained in
its support. The assertion extends to the closed movable cone
by continuity.

For curves in a general abelian boundary fiber of the Satake
morphism, $\lambda\cdot C=0$ and $D\cdot C<0$; on the abelian
cover the normal class of the boundary is $-2\Theta$.
Thus every movable class $a\lambda-b[D]$ has $b\ge0$.
Next take a general curve $C$ in the family contracted by $f$
and covering $E_r$. Since $\lambda_r$ is movable,
$\lambda_r\cdot C\ge0$. In fact it is positive. If it were zero,
the inequality
\[
 E_r\cdot C=14\lambda_r\cdot C-2D_r\cdot C<0
\]
would give $D_r\cdot C>0$. An ample class
$a_0\lambda-b_0[D]$ on $X_0$ has $b_0>0$, by the preceding
boundary-fiber test. Its transform is movable on $X_r$ but would
have negative degree on $C$, a contradiction.

Set
\[
 \sigma=\frac{D_r\cdot C}{\lambda_r\cdot C}.
\]
This number is rational and is greater than seven. The curve test
on $E_r$ forces $a\ge\sigma b$ for every movable class. The ray
$L_*$ is therefore generated by $\sigma\lambda_r-D_r$.
All the nef cones in \eqref{eq:finite-MMP} lie in the closed
sector bounded by $\lambda$ and $L_*$. Intersect this sector
with a positive affine hyperplane in $N^1(X_0)_\mathbb R$.
The finitely many nef cones give closed intervals with adjacent
intervals intersecting. Their union is connected and contains
both endpoints of the sector, so it is the whole interval.
This proves that their union is the closed movable cone and
verifies the Mori dream space definition.

The boundary ray $L_*$ is big and semiample, so a sufficiently
divisible multiple is moving. Hence its slope $\sigma$ is the
attained moving slope. The upper bound in
\eqref{eq:movable-cone} is Theorem~\ref{thm:moving-form}.
\end{proof}

\subsection{The algebra of prescribed boundary orders}
Let $t$ be a formal variable and retain both gradings in
\begin{equation}\label{eq:boundary-algebra}
 \mathcal R_D=
 \bigoplus_{k,b\ge0}\MM_k^{(6)}[b]t^b.
\end{equation}
This algebra records all scalar weights and all required boundary
orders at once. Finite generation of the separate modules
$\mathcal M_\ell$ does not by itself imply finite generation of
\eqref{eq:boundary-algebra}.

\begin{theorem}\label{thm:boundary-algebra}
The algebra $\mathcal R_D$ is finitely generated over $\C$.
Its associated graded algebra
\[
 \operatorname{gr}_D\MM_\bullet^{(6)}
 =\bigoplus_{k,b\ge0}
     \MM_k^{(6)}[b]/\MM_k^{(6)}[b+1]
\]
is a finitely generated domain. The first Fourier--Jacobi coefficient
identifies it with a subalgebra of
$\bigoplus_{k,b\ge0}\JJ_{k,b}^{(5)}$, where index zero means
scalar Siegel modular forms of genus five.
\end{theorem}
\begin{proof}
By Theorem~\ref{thm:mori-dream} and \cite[Proposition~2.9]{HK00},
the Cox ring of $X_0$ is finitely generated. Choose a positive
integer $N$ such that $N\lambda$ and $ND$ are Cartier and the
Hodge power descends to the coarse space. Enlarge $N$ so that
these classes also belong to the lattice used to grade the Cox ring. Then
\[
 \mathcal R_D^{(N)}
   =\bigoplus_{k,b\ge0}\MM_{Nk}^{(6)}[Nb]t^{Nb}
   =\bigoplus_{k,b\ge0}H^0(X_0,Nk\lambda-NbD)t^{Nb}
\]
is finitely generated. To justify this directly, choose homogeneous
Cox generators. The nonnegative exponent vectors of monomials
whose degrees lie in the required lattice and cone form a finitely
generated monoid by Gordan's lemma. Their images generate precisely
the indicated subalgebra. This is a restriction to a rational
polyhedral cone and a finite-index lattice, not an assertion about
arbitrary subalgebras of a finitely generated ring.

We must also include weights and orders not divisible by $N$.
The finite group $\mu_N^2$ acts on $\mathcal R_D$ through its two
gradings, with invariant ring $\mathcal R_D^{(N)}$.
Every element is integral over the invariant ring, by its orbit
polynomial. Moreover
\[
 [\operatorname{Frac}(\mathcal R_D):
       \operatorname{Frac}(\mathcal R_D^{(N)})]\le N^2.
\]
Here the invariant fraction field equals the fraction field of the
invariants, as follows by multiplying a denominator by its other
translates. The integral closure of the finitely generated complex
domain $\mathcal R_D^{(N)}$ in this finite field extension is a
finite module. It contains $\mathcal R_D$, which is consequently
a finite module as well, since the invariant ring is Noetherian.
This proves finite generation of the full algebra.

The associated graded algebra is the quotient by the homogeneous
ideal $\bigoplus_{k,b\ge0}\MM_k^{(6)}[b+1]t^b$.
The map $G\mapsto\phi_b(G)$ on its $(k,b)$ component is injective
and respects multiplication. The boundary order is additive on
nonzero forms, so this associated graded algebra is a domain.
\end{proof}

\begin{remark}
The image in the Jacobi algebra consists of coefficients that extend
to genus-six Siegel modular forms with the prescribed preceding
coefficients zero. In particular, a positive-index coefficient has
boundary order at least its index in the genus-five variables.
There is no assertion that every Jacobi form occurs. This distinction
is essential in comparison with \cite{BBHJ24}, which proves that
the full rings of Siegel--Jacobi forms of fixed or bounded positive
relative index, for neat groups, are not finitely generated.
The proof above uses neither those full rings nor the incorrect
finite-generation statement discussed in \cite[Remark~1.3]{BR25}.
\end{remark}

\subsection{The minimal model and its boundary}
\begin{proposition}\label{prop:unique-minimal-model}
Every birational map between projective $\Q$-factorial terminal
minimal models of $\A_6$ is an isomorphism. In particular, $Y$ is
unique up to isomorphism, and
\[
 \operatorname{Bir}(\A_6)=\operatorname{Aut}(Y).
\]
\end{proposition}
\begin{proof}
By \cite[Theorem~1]{Kawamata08}, a birational map between two such
minimal models is a composition of flops. No flop can start from
$Y$, since $\rho(Y)=1$: a nontrivial projective birational
contraction would pull back an ample divisor to a nonzero nef
class which is not ample. Thus every minimal model is isomorphic
to $Y$, and the same argument applies to every birational self-map.
\end{proof}

\begin{theorem}\label{thm:minimal-boundary}
Let $D_Y$ be the strict transform of $D$ on $Y$. Then $D_Y$ is
a prime ample unirational $\Q$-Cartier divisor, and
\begin{equation}\label{eq:common-affine-open}
 U:=\A_6\setminus V(\F)\cong Y\setminus D_Y.
\end{equation}
The center on $Y$ of the valuation $E$ is contained in $D_Y$.
Furthermore,
\begin{equation}\label{eq:boundary-threshold}
 \operatorname{lct}(Y;D_Y)
   \le\frac{3(\sigma-7)}{\sigma}\le\frac3{43}.
\end{equation}
In particular the reduced pair $(Y,D_Y)$ is not log canonical.
\end{theorem}
\begin{proof}
The cusp form $\F$ vanishes on the Satake boundary. Thus
$U=D_+(\F)$ in the Satake compactification
$\operatorname{Proj}\MM_\bullet^{(6)}$, and $U$ is affine.
On $X_0$ it is $X_0\setminus(E\cup D)$.

We use a consequence of normality and separatedness. If an affine
open subset $U$ of a normal separated variety $V$ has complement
of codimension at least two, then $U=V$. Indeed, restriction gives
$\Gamma(V,\OO_V)=\Gamma(U,\OO_U)$ and hence a retraction $V\to U$.
Its composition with the open immersion agrees with the identity
on a dense open subset, and therefore everywhere.

Every curve contracted by a flipping contraction in
\eqref{eq:finite-MMP} is contained in $E_i$, because
$E_i\cdot C=2K_{X_i}\cdot C<0$. Hence the flip embeds $U$ into
$X_{i+1}\setminus(E_{i+1}\cup D_{i+1})$ with complement of
codimension at least two. The preceding observation proves equality.
For the last contraction put $Z=f(E_r)$, which has codimension
at least two in $Y$. Away from its exceptional locus, $f$ gives
\[
 U=Y\setminus(D_Y\cup Z).
\]
Apply the observation to $V=Y\setminus D_Y$. It follows that
$Z\subset D_Y$ and proves \eqref{eq:common-affine-open}.
The divisor $D_Y$ is nonzero, effective and $\Q$-Cartier. Since
$\rho(Y)=1$, it is ample. It is birational to $D$, which is
unirational by \cite[Theorem~0.1]{FV16}.

There is only one $f$-exceptional prime divisor, so write
\[
 f^*D_Y=D_r+cE_r.
\]
The class $\sigma\lambda_r-D_r$ is pulled back from $Y$.
For an $f$-contracted curve $C$, one therefore has
\[
 D_r\cdot C=\sigma\lambda_r\cdot C,\qquad
 E_r\cdot C=2(7-\sigma)\lambda_r\cdot C<0.
\]
Intersecting the pullback identity with $C$ gives
\begin{equation}\label{eq:boundary-valuation}
 \ord_E(D_Y)=c=\frac{\sigma}{2(\sigma-7)}.
\end{equation}
The quantity on the left is a divisorial valuation, not an ordinary
multiplicity at a point. Since $a(E,Y)=1/2$, its log discrepancy
for the pair $(Y,tD_Y)$ is
\[
 A_{Y,tD_Y}(E)=1+a(E,Y)-t\ord_E(D_Y)
            =\frac32-t\frac{\sigma}{2(\sigma-7)}.
\]
By definition the log canonical threshold is the infimum of
$A_Y(v)/v(D_Y)$, where $A_Y(v)=1+a(v,Y)$, over divisorial valuations with $v(D_Y)>0$.
Testing $v=\ord_E$ gives the first inequality in
\eqref{eq:boundary-threshold}, and
$7<\sigma\le43/6$ gives the second. No equality for the threshold
is asserted.
\end{proof}

The affine open set in \eqref{eq:common-affine-open}, together with
its boundary valuation, determines the minimal model algebraically. Put $B=\Gamma(U,\OO_U)$ and let
$\nu_D$ be the boundary valuation of its fraction field. Then
\begin{equation}\label{eq:intrinsic-filtration}
 \begin{split}
 B_m&=\{h\in B:\nu_D(h)\ge-m\}\\
 &=\bigcup_{r\ge0}
 \left\{\frac{G}{\F^r}:
 G\in\MM_{14r}^{(6)}[\max\{0,2r-m\}]\right\}
       \qquad(m\ge0).
 \end{split}
\end{equation}
The union is increasing by multiplication of numerator and
denominator by $\F$. Normality of $Y$ and
\eqref{eq:common-affine-open} identify
$B_m=H^0(Y,\OO_Y(mD_Y))$. Hence, for $N$ such that $ND_Y$ is
Cartier,
\begin{equation}\label{eq:intrinsic-minimal-model}
 Y\cong\operatorname{Proj}\bigoplus_{m\ge0}B_{Nm}.
\end{equation}
This description uses only modular functions on $U$ and their
boundary orders; it does not require the individual flips to be
constructed explicitly.

\begin{corollary}\label{cor:rational-fibrations}
Let $\A_6\dashrightarrow B_0$ be a dominant rational map with
connected general fiber and $0<\dim B_0<21$. A smooth projective
model $V$ of its general fiber is of general type:
\[
 \kappa(V)=\dim V=21-\dim B_0.
\]
\end{corollary}
\begin{proof}
We give the argument of \cite[Theorem~2.1 and Section~2, Remark~2]{AC08}
in the present setting. Replace the base by a smooth projective model and resolve the map
as $\pi:W\to Y$ and $h:W\to B_0$, with $W$ smooth. For a general member $H$ of a very
ample linear system on $B_0$, the divisor $L=\pi_*h^*H$ is
nonzero, effective and $\Q$-Cartier, hence ample on $Y$.
After replacing $L$ and $H$ by corresponding Cartier multiples,
resolving the base ideal gives
\[
 \pi^*L=h^*H+\sum_i b_iE_i,\quad b_i\ge0,
 \qquad K_W\sim_\Q\sum_i a_iE_i,\quad a_i>0.
\]
Here the $E_i$ run through all $\pi$-exceptional prime divisors,
and some $b_i$ may be zero. The second formula uses terminality
and $K_Y\sim_\Q0$.
For a smooth general fiber $V$ of $h$, adjunction gives
$K_W|_V=K_V$. Such a fiber is not contained in any $E_i$, so
each restriction $E_i|_V$ is effective or zero. The divisor $\pi^*L|_V$
is big, since $\pi|_V$ is generically finite onto its image.
Choose $M>0$ with $Ma_i\ge b_i$ for every $i$. Restriction gives
\[
 MK_V-\pi^*L|_V
       \sim_\Q\sum_i(Ma_i-b_i)E_i|_V\ge0.
\]
Thus $K_V$ is big.
\end{proof}

\section{Birational automorphisms}
\label{sec:automorphisms}

We retain the notation of \eqref{eq:finite-MMP}. For a normalized divisorial
valuation $v$, write $a(v,X)$ for its discrepancy and
$A_X(v)=1+a(v,X)$ for its log discrepancy. We prove that the
valuation defined by $E$ can be recovered from $Y$ alone.
The order-ten vanishing in the proof of Lemma~\ref{lem:origin-ten}
is needed at the finite quotient singularities; terminality by
itself would not suffice.

\subsection{Finite stabilizers and toroidal charts}

For a finite-order linear transformation $g$ with eigenvalues
$\exp(2\pi i\alpha_j)$, where $0\le\alpha_j<1$, put
$\operatorname{age}(g)=\sum_j\alpha_j$.

\begin{lemma}\label{lem:age-six}
Let $V$ be a complex vector space of dimension six, and let $g\in\GL(V)$
have finite order. Suppose that the characteristic polynomial of
$g$ on $V\oplus\overline V$ has integral coefficients. If $g$
acts nontrivially on $\Sym^2V$, then
\[
 \operatorname{age}(g\mid\Sym^2V)\ge\frac76.
\]
The only case in which this age is at most $3/2$ is, up to
multiplication of all eigenvalues on $V$ by $-1$,
\[
 (\exp(2\pi i/6),1,1,1,1,1).
\]
Its age on $\Sym^2V$ is $7/6$.
\end{lemma}
\begin{proof}
Suppose that a cyclotomic factor $\Phi_m$ of degree $2d\ge4$
occurs on $V\oplus\overline V$. The eigenangles on $V$ then
contain a set $B=\{\beta_1,\ldots,\beta_d\}$ obtained by choosing
one member from each pair $\{a/m,1-a/m\}$ with
$1\le a<m/2$ and $(a,m)=1$. This remains true when the factor
occurs with multiplicity. Put
\[
 S(B)=\sum_{i\le j}\{\beta_i+\beta_j\},\qquad
 c(B)=\min_j\sum_i\{\beta_i-\beta_j\},
\]
where braces denote fractional parts. The minimum of
$\sum_i\{\beta_i+t\}$ occurs immediately after one of its
downward jumps, namely at $t=-\beta_j$. Thus every additional
eigenangle contributes at least $c(B)$ in its cross terms with
$B$, and the required age is at least $S(B)+(6-d)c(B)$.
Put $\epsilon=\gcd(2,m)/m$. The off-diagonal sums
$\{\beta_i+\beta_j\}$ are nonzero multiples of $\epsilon$.
For each fixed $j$, the $d-1$ nonzero differences
$\{\beta_i-\beta_j\}$ are distinct multiples of $\epsilon$;
thus $c(B)\ge\epsilon\binom d2$. The diagonal terms give
the bound
\[
 S(B)+(6-d)c(B)
 \ge D_m+(7-d)\binom d2\epsilon,\qquad
 D_m=\sum_{\substack{1\le a<m/2\\(a,m)=1}}
       \min\bigl(\{2a/m\},\{-2a/m\}\bigr).
\]
Let $\varphi$ denote Euler's totient function. The elementary
sums $D_m$ give the following bounds:
\[
\begin{array}{c|l|c}
 d&m\text{ with }\varphi(m)=2d&
 D_m+(7-d)\binom d2\epsilon\ \ge\\ \hline
 2&5,8,10&8/5\\
 3&7,9,14,18&19/9\\
 4&15,16,20,24,30&32/15\\
 5&11,22&35/11\\
 6&13,21,26,28,36,42&15/7
\end{array}
\]
For the remaining case $m=12$, the two angles are chosen from
$\{1/12,11/12\}$ and $\{5/12,7/12\}$. Their diagonal terms
sum to at least $1/3$, their cross term is at least $1/3$, and
$c(B)\ge1/3$. Hence $S(B)+4c(B)\ge2$.
Thus every cyclotomic factor of degree at least four gives
age greater than $3/2$.

It remains to consider angles in
$\{0,1/6,1/4,1/3,1/2,2/3,3/4,5/6\}$. Let $s$ be the number
outside $\{0,1/2\}$. Their diagonal terms contribute at least
$s/3$, and their cross terms with the other angles at least
$s(6-s)/6$. If $2\le s\le6$, the resulting bound
$s(8-s)/6$ is at least two. If $s=0$, a mixture of $p$ zero
angles and $6-p$ half angles contributes $p(6-p)/2\ge5/2$.
If $s=1$ and the five central angles are mixed, their cross
terms already contribute at least two. Otherwise multiplication
by $-1$ makes those five angles zero. The remaining contribution
$\{2\alpha\}+5\alpha$ is at most $3/2$ only for
$\alpha=1/6$, when it equals $7/6$.
\end{proof}

\begin{lemma}\label{lem:toroidal-age-height}
Locally analytically, $X_0$ admits proper birational models
$Q=Z/G$, where $Z$ is a smooth equivariant toroidal refinement
and $G$ is finite. The quotient map $Z\to Q$ is unramified in
codimension one. At a fixed point of a nonidentity stabilizer
element, its tangent age is greater than $3/2$, except possibly
when the semiabelian variety has a principally polarized elliptic
factor with $j=0$ on which the element acts by $\exp(2\pi i/6)$,
and it acts identically on the complementary semiabelian factor.
Every prime divisor of $Q$ exceptional over $X_0$ has discrepancy
greater than $1/2$ over $X_0$.
\end{lemma}
\begin{proof}
We use the toroidal local description of
\cite[\S1, Lemma~1.2]{SB16}. Write $h+r=6$, let $C$ be the
abelian quotient of dimension $h$, let
$W=H^0(C,\Omega_C^1)^\vee$, and let $\Lambda\cong\Z^r$ be
the toric lattice. The tangent representation of the smooth
base is
\[
 \Sym^2W\oplus(W\otimes\Lambda_\C).
\]
A finite-group-equivariant smooth subdivision exists, as in
\cite[p.~2]{SB16}. At a fixed point, the smooth cone associated
with its torus orbit has a saturated lattice generated by its
primitive rays. The group permutes these rays. A cycle of length
$l$ on their normal coordinates is a weighted permutation with
age $(l-1)/2+\delta$, $0\le\delta<1$. The corresponding
permutation representation has age $(l-1)/2$. The tangent to
the orbit is the quotient lattice representation; translations
have no effect on its logarithmic derivative. Splitting the
tangent exact sequences as finite-group representations gives
\begin{equation}\label{eq:toroidal-age-bound}
 \operatorname{age}(g\mid T_Z)
 \ge \operatorname{age}\bigl(g\mid
                   \Sym^2(W\oplus\Lambda_\C)\bigr).
\end{equation}
The characteristic polynomial on
$W\oplus\overline W\oplus2\Lambda_\C$ is integral.

For $h\ge1$, effectivity is precisely Lemma~1.2 of \cite{SB16}:
the only kernel on the pair of lattice representations is the
simultaneous central sign, which is already removed from the
effective modular group. When $h=0$, a maximal-cusp parabolic
element acts by
\[
 \Omega\longmapsto A\Omega A^t+BA^t,
 \qquad A\in\GL_r(\Z),\quad BA^t\in\Sym_r(\Z).
\]
The full unipotent translation lattice is $\Sym_r(\Z)$.
After quotienting by this lattice, the integral translation in
the displayed formula acts trivially on the exponential torus. Its effective action is therefore the lattice
action of $A$, whose kernel on $\Sym^2\Lambda_\C$ is
$\{\pm I\}$. There is no additional translation kernel.
Lemma~\ref{lem:age-six} now applies to every nonidentity
effective stabilizer. In particular all its tangent ages exceed
one, so there are no quasireflections and the quotient is
unramified in codimension one.

In the exceptional spectrum of Lemma~\ref{lem:age-six}, the
rational representation $\Lambda_\C$ cannot contain the single
nonreal eigenvalue. After the simultaneous central sign has
been removed, the lattice action is the identity and the action
on $W$ has eigenvalues $\exp(2\pi i/6),1,\ldots,1$.
For the resulting automorphism $g$ of $C$, the endomorphisms
\[
 e=g-g^2,\qquad 1-e=g^2-g+1
\]
are integral complementary idempotents. They are self-adjoint
for the Rosati involution, because $g^\dagger=g^{-1}$.
Consequently $C$ splits as a product of principally polarized
abelian varieties $C=E_{\mathrm{ell}}\times C'$, with $j(E_{\mathrm{ell}})=0$.
All extension classes indexed by $\Lambda$ are fixed by $g$.
On $E_{\mathrm{ell}}$, the endomorphism $g-1$ is a unit, so these extension
classes have zero $E_{\mathrm{ell}}$-component. The semiabelian variety itself
therefore splits as $E_{\mathrm{ell}}$ times its complement. This argument
is in the category of extensions and also removes any Jacobi
translation in the $E_{\mathrm{ell}}$-component; no splitting up to isogeny
is being used.

We finally check the toric discrepancies with the lattice
convention made explicit. The integral translations of $\Omega$
give the cocharacter lattice
\[
 N=\Sym_r(\Z),\qquad
 M=N^\vee=\{T=T^t:T_{ii}\in\Z,\ 2T_{ij}\in\Z\},
 \qquad\langle T,B\rangle=\tr(TB).
\]
Thus the primitive rays of a perfect cone are $vv^t$ with $v$
primitive integral. Let $P$ be its positive definite supporting
form, normalized by $\min_{v\ne0}v^tPv=1$, and write
$h_P(u)=\tr(Pu)$. A newly introduced ray has rank $s\ge2$,
because a rank-one positive semidefinite matrix spans an
extremal ray of the whole positive semidefinite cone. An integral
change of basis writes its primitive generator as
$u=\operatorname{diag}(U,0)$ with $U$ positive definite integral;
in particular $\det U\ge1$. The restriction $P_s$ has lattice
minimum at least one. Minkowski's convex body theorem gives
\[
 \det P_s\ge\left(\frac{v_s}{2^s}\right)^2,
 \qquad v_s=\frac{\pi^{s/2}}{\Gamma(1+s/2)}.
\]
Hence the arithmetic--geometric mean inequality gives
\begin{equation}\label{eq:toric-height-gap}
 h_P(u)=\tr(P_sU)\ge s(\det P_s\det U)^{1/s}
        \ge\frac{s}{4}v_s^{2/s}>\frac32
        \qquad(2\le s\le6).
\end{equation}
The toric discrepancy is $h_P(u)-1$. The finite quotient has no
divisorial ramification, so this also proves the last assertion.
\end{proof}

\subsection{The exceptional valuation of discrepancy one half}

\begin{theorem}\label{thm:unique-small-discrepancy}
Among exceptional divisorial valuations over $Y$, the valuation
defined by $E$ is the unique one of discrepancy less than one.
More precisely,
\[
 a(E,Y)=\frac12,\qquad a(v,Y)\ge1\quad\text{for }v\ne\ord_E.
\]
The second bound is attained.
\end{theorem}
\begin{proof}
Use throughout the rational bicanonical differential defined by
$\F$. Its divisor on $X_0$ is $E$ and its divisor on $Y$ is zero.
Every exceptional valuation over $Y$, other than $E$, is
exceptional over $X_0$, and
\begin{equation}\label{eq:discrepancy-with-modular-divisor}
 a(v,Y)=a(v,X_0)+\frac12v(E).
\end{equation}
Since $2K_Y\sim0$, all these discrepancies are half-integral: on a smooth model
$p:W\to Y$, the divisor $2K_W-p^*(2K_Y)$ is integral.
It is enough to prove $a(v,Y)>1/2$ for $v\ne\ord_E$.

Take a local model $Q=Z/G$ as in
Lemma~\ref{lem:toroidal-age-height}, and let $x_\rho$ be the
boundary parameters on $Z$. In an ordinary, nowhere vanishing
canonical frame on $Z$, the bicanonical differential has
coefficient
\begin{equation}\label{eq:bicanonical-coefficient}
 b=\varepsilon\,\frac{\F}{\prod_\rho x_\rho^2},
 \qquad\varepsilon\in\OO_Z^\times.
\end{equation}
Indeed the period differential is, up to a unit, the wedge of
the base differentials and logarithmic torus differentials.
The function $b$ need not be invariant: the bicanonical
differential is invariant, so $b$ has the inverse square of
the determinant character in linearized coordinates.

The coefficient $b$ is holomorphic. To see this directly, every
nonzero Fourier index $T$ of $\F$ satisfies $v^tTv\ge2$ for
every primitive integral $v$, because the full modular group
has a single boundary divisor and the boundary order is two.
For a ray $u=\sum_i a_i v_iv_i^t$ in a perfect cone this gives
$\ord_u\F\ge2\sum_i a_i=2h_P(u)$. Writing $D_u$ for the boundary prime corresponding to $u$, we get
$\ord_{D_u}b\ge2(h_P(u)-1)\ge0$ on every boundary component.
This also handles all prime divisors on $Q$ exceptional over
$X_0$, by \eqref{eq:toric-height-gap}.

Now let $v$ be exceptional over $Q$, let $w$ be a normalized
extension to $Z$, and let $e$ be its ramification index, so that
$w|_{\C(Q)}=ev$. The
inertia group is cyclic of order $e$. The usual valuation
ramification formula, applied to the bicanonical differential,
is
\begin{equation}\label{eq:inertia-bicanonical}
 A_Y(v)=\frac{A_Z(w)+w(b)/2}{e}.
\end{equation}
Indeed the pullback to a resolution above $Z$ has order
$w(b)+2(A_Z(w)-1)$, whereas ramification from a resolution
above $Q$ gives order $2e(A_Y(v)-1)+2(e-1)$.
For the cyclic inertia and this ramification calculation see
\cite[\S\S2.5--2.6]{IR96}, whose local argument applies to
arbitrary finite subgroups of $\GL_n$.

If $e=1$, the valuation is exceptional over smooth $Z$, so
$A_Z(w)\ge2$ and \eqref{eq:inertia-bicanonical} suffices.
If $e>1$, choose the inertia generator acting on a uniformizer
by $\exp(2\pi i/e)$. In eigenparameters $t_j$ with eigenvalues
$\exp(2\pi i a_j/e)$, where $0\le a_j<e$, the fact that inertia acts trivially on
the residue field gives
\[
 w(t_j)\equiv a_j\pmod e,\qquad
 \frac{A_Z(w)}e\ge\sum_j\frac{w(t_j)}e
                   \ge\sum_j\frac{a_j}e.
\]
For the first inequality, take an equivariant system of parameters
at a general smooth point of the center of $w$. Parameters
tangent to the center have value zero. Pulling their top wedge
to a model carrying $w$ gives order at least
$\sum_jw(t_j)-1$, and hence the displayed Jacobian inequality.
If the tangent age is greater than
$3/2$, holomorphicity of $b$ finishes the proof.

In the remaining case, Lemma~\ref{lem:toroidal-age-height}
provides the elliptic product locus on the smooth chart. Its
ideal $I$ has five normal parameters: $h-1$ cross-period
parameters and $r$ extension parameters. They all have
eigenvalue $\exp(2\pi i/6)$. The Taylor argument in the proof
of Lemma~\ref{lem:origin-ten} gives $\F\in I^{10}$ on the
interior. Locally the toroidal chart is a torus embedding over
the deformation and extension base, and its fan depends only
on $\Lambda$. Thus its boundary parameters vary independently
of these five normal parameters. The ideal $I^{10}$ is
saturated with respect to the product of the boundary parameters,
so the same inclusion
holds on the entire chart. Applying this observation once more to
\eqref{eq:bicanonical-coefficient} gives $b\in I^{10}$.
All these ideals are on the smooth chart, before the finite
quotient is taken. Averaging lifts from $I/I^2$ supplies
equivariant normal parameters. Consequently
\[
 \frac{w(b)}e\ge\frac{10}{6},\qquad
 A_Y(v)\ge\frac76+\frac12\frac{10}{6}=2.
\]
This proves the asserted discrepancy bound in every case.

For sharpness, take a general product $E_{\mathrm{ell}}\times A$ with
$j(E_{\mathrm{ell}})=0$ and $\dim A=5$. The transverse quotient has type
$\frac16(1,1,1,1,1,2)$. Its primitive weighted ray defines
a valuation with $a(v,X_0)=1/6$. By
Corollary~\ref{cor:product-normal-term}, the degree-ten normal term
of $\F$ is
\[
 \Delta(\omega)^2T_{10}(\tau,z).
\]
Since $\Delta$ is nonzero on $\Hh_1$ and
$T_{10}\ne0$, this valuation has $v(E)=10/6$. Equation
\eqref{eq:discrepancy-with-modular-divisor} gives $a(v,Y)=1$.
\end{proof}

\begin{theorem}\label{thm:birational-automorphisms}
The birational automorphism group of $\A_6$ is trivial.
\end{theorem}
\begin{proof}
By Proposition~\ref{prop:unique-minimal-model}, a birational
automorphism is an automorphism of $Y$.
Theorem~\ref{thm:unique-small-discrepancy} shows that it fixes
the valuation of $E$. Since $-E_r$ is ample over $Y$, the last
model $X_r$ is recovered as the relative Proj of the valuation
ideal sheaves
\[
 \mathcal I_m=f_*\OO_{X_r}(-mE_r),
 \qquad
 X_r\cong\Proj_Y\bigoplus_{m\ge0}\mathcal I_{m\ell}
\]
for sufficiently divisible $\ell>0$. These sheaves are intrinsic
to the valuation: on an open set meeting its center their
sections are the regular functions with $\ord_E(a)\ge m$,
and outside the center they equal $\OO_Y$. Thus the automorphism
lifts to $X_r$. It acts trivially on $N^1(X_r)$: it fixes
$[E_r]$ and $f^*N^1(Y)$, which are independent, and it acts
trivially on the latter one-dimensional space because it
preserves its integral lattice and ample ray.

Conjugating by the flips gives a pseudoautomorphism of $X_0$
acting trivially on $N^1(X_0)$. Since
$\Pic(X_0)_\Q=N^1(X_0)_\Q$, it preserves some positive
power of an ample line bundle. Sections of its powers are
unchanged under an isomorphism in codimension one on a normal
variety. Taking their Proj therefore makes the
pseudoautomorphism regular. It also preserves a positive
multiple of $\lambda$, hence the Satake contraction. The
unique exceptional divisor $D$ of that contraction is preserved.
We have now proved, rather than assumed, that the automorphism
preserves the modular open subset $\A_6$.

We finish by recalling a direct argument that this open subset
has no automorphisms. Put $\Gamma=\Sp_{12}(\Z)/\{\pm I\}$.
By Lemma~\ref{lem:age-six} there are no quasireflections, so
the inverse image $\Hh_6^\circ$ of the smooth locus of $\A_6$
is the locus on which $\Gamma$ acts freely. Its complement
is a locally finite union of fixed loci of complex codimension
at least two. General position makes $\Hh_6^\circ$ simply
connected. An automorphism of $\A_6$ therefore lifts to
$\Hh_6^\circ$. In a bounded convex realization of the Siegel
domain, Hartogs extension extends its coordinate functions
across the omitted fixed loci. A supporting-hyperplane maximum
principle shows that the image stays in the domain. The inverse
extends in the same way. The lift is therefore a holomorphic
automorphism of $\Hh_6$ normalizing $\Gamma$.

Represent this lift by $S\in\Sp_{12}(\mathbb R)$.
For primitive $v\in\Z^{12}$, put $T_v=I+vv^tJ$, with $J$
the standard symplectic matrix. Since
$ST_vS^{-1}=T_{Sv}$ is integral, $Sv=\sqrt c\,w$ for a
primitive integral $w$ and a positive integer $c$.
If $c>1$, conjugating $T_w$ back gives an integral unipotent
$c$-th root of $T_v$. The unique unipotent root is
$I+c^{-1}vv^tJ$, which is not integral since $v$ is primitive.
Thus $c=1$. Applying this to the standard basis gives
$S\in\Sp_{12}(\Z)$, and the induced automorphism of
$\A_6$ is the identity.
\end{proof}

\appendix
\section{The exact theta certificate}
\label{sec:theta-certificate}

\subsection{The theta-jet matrix}

Put $e(x)=\exp(2\pi i x)$.  For $\mu\in(\Z/4\Z)^5$, define
\begin{equation}\label{eq:theta-basis}
 \theta_\mu(\tau,z)
  =\sum_{r\equiv\mu\pmod4}
       e\!\left(\frac{r^t\tau r}{8}+r^tz\right).
\end{equation}
These functions form a basis of $H^0(A_\tau,4\Theta)$.  Choose one
representative from each orbit of $\mu\mapsto-\mu$.  For an orbit of
size two, put
\[
 b_\mu=\frac{\theta_\mu+\theta_{-\mu}}2,
\]
and for an orbit of size one put $b_\mu=\theta_\mu$.
There are $32$ orbits of size one, and hence
\begin{equation}\label{eq:even-dimension}
 \dim H^0(A_\tau,4\Theta)^+
   =\frac{4^5+2^5}{2}=528.
\end{equation}
For a multi-index $\alpha\in\Z_{\geq0}^5$ of even total degree,
the $\alpha$-derivative of $b_\mu$ at the origin equals that of
$\theta_\mu$.  Indeed the substitution $r\mapsto-r$ leaves $r^\alpha$
unchanged.

The map taking an even section to its Taylor coefficients of degrees
$0,2,4,6,8$ has $791$ coordinates, since
\[
 \sum_{j=0}^{4}\binom{2j+4}{4}=1+15+70+210+495=791.
\]
Let $J(\tau)$ be its $791\times528$ matrix in the basis just defined.
The nonzero section $\phi$ and Lemma~\ref{lem:origin-ten} imply
\begin{equation}\label{eq:jet-rank-upper}
 \rk J(\tau)\leq527\qquad\text{for every }\tau\in\Hh_5.
\end{equation}
To justify the assertion at points where the entire section $\phi$ is
zero, observe that every $528\times528$ minor vanishes on the nonempty
open set where $\phi$ is nonzero.  It therefore vanishes identically.

For the finite calculation we restrict to $\tau=uQ$, where $u\in\Hh_1$
and
\begin{equation}\label{eq:Q-matrix}
 Q=\begin{pmatrix}
 9&1&0&-1&2\\
 1&10&2&1&-1\\
 0&2&11&-2&1\\
 -1&1&-2&12&0\\
 2&-1&1&0&13
 \end{pmatrix}.
\end{equation}
The matrix is positive definite; in fact $Q\geq5I_5$.
We use $t=e(u)$ as the variable for this calculation, to distinguish it
from the Fourier--Jacobi boundary variable of a degree-six modular form.
Set
\[
 d_\mu=\min_{r\equiv\mu\pmod4}r^tQr.
\]
The difference of the norms of any two vectors in the same coset is
divisible by $8$, and $d_\mu=d_{-\mu}$.  Consequently the moments of the
normalized sections $t^{-d_\mu/8}b_\mu(uQ,z)$ are integer power series.
Multiplying the row of Taylor coefficients indexed by $\alpha$ by
the nonzero scalar $\alpha!/(2\pi i)^{|\alpha|}$ gives the matrix
\begin{equation}\label{eq:jet-integral-matrix}
 A_{\alpha,\mu}(t)
  =\sum_{r\equiv\mu\pmod4}
      r^\alpha t^{(r^tQr-d_\mu)/8}
       \in\Z[[t]],\qquad |\alpha|\in\{0,2,4,6,8\}.
\end{equation}
All its $528\times528$ minors vanish identically by
\eqref{eq:jet-rank-upper}.  Notice that this is an exact identity over
$\Z[[t]]$, obtained before any reduction modulo a prime.
For the numerical ordering, represent $\mu_i$ by $0,1,2,3$, order cosets
by $\sum_{i=1}^5 4^{i-1}\mu_i$, and choose the smaller member of each
pair $\{\mu,-\mu\}$.  Order rows first by total degree and then
lexicographically by their multi-indices.

\subsection{Lifting a projective kernel by maximal minors}

The following elementary observation avoids any assumption that a
normalized kernel vector lifts integrally.

\begin{lemma}\label{lem:cofactor-lift}
Let $p$ be a prime and let $A(t)$ be an $n\times m$ matrix over
$\Z_{(p)}[[t]]$.  Suppose that every $m\times m$ minor of $A(t)$ is zero
and that its reduction $\overline A(t)$ has rank $m-1$ over
$\Ff_p((t))$.  Then there is a vector
\[
 v(t)\in\ker A(t)\subset\Z_{(p)}[[t]]^m
\]
whose reduction is nonzero.  If $\overline K(t)$ spans
$\ker\overline A(t)$, then $\overline v(t)=c(t)\overline K(t)$ for
some $c(t)\in\Ff_p((t))^\times$.

In particular, let $P$ be a homogeneous polynomial in the coordinates
of a vector, with coefficients in $\Z_{(p)}[[t]]$.  If
$\overline P(\overline K(t))\ne0$, then $P(v(t))\ne0$.
\end{lemma}

\begin{proof}
Choose $m-1$ rows forming a matrix $B(t)$ whose reduction has rank
$m-1$.  For $1\leq j\leq m$, let
\[
 v_j(t)=(-1)^{j-1}\det B_{\widehat j}(t),
\]
where $B_{\widehat j}$ denotes the matrix obtained by deleting column
$j$.  Some $v_j$ is nonzero modulo $p$.  For any row of $A$, its scalar
product with $v$ is, up to sign, the determinant obtained by adjoining
that row to $B$.  This is zero by hypothesis.  Hence $Av=0$.
Since $\ker\overline A$ is one-dimensional, the asserted proportionality
follows.  If $P$ has degree $h$, then
\[
 \overline P(\overline v)
     =c(t)^h\overline P(\overline K)\ne0,
\]
which implies $P(v)\ne0$ in characteristic zero.
\end{proof}

\subsection{The finite certificate}

We describe both the enumeration bound and the matrix operations, so
that the calculation can be reproduced without earlier numerical data.
The ancillary program \texttt{certificate.py} uses only Python and
NumPy.  All its operations are integer operations or operations in
$\Ff_{101}$.  The complete pivot selections are recorded in
\texttt{certificate.json}.
An optional second implementation,
\texttt{independent\_\allowbreak certificate\_\allowbreak check.py},
uses block-Toeplitz
matrices and polynomial interpolation to verify the same ranks and
precision bounds. The first program alone reproduces the certificate
described below.

First, all coset minima satisfy $d_\mu\leq140$.  This bound can already
be verified by the $5^5$ vectors in $[-2,2]^5$.  Thus the coefficients
of \eqref{eq:jet-integral-matrix} through $t^{14}$ involve only vectors
with $r^tQr\leq252$.  This enumeration is complete in $[-5,5]^5$.
Indeed
\[
 \det Q=130217,
\]
and the five diagonal cofactors of $Q$ are
\[
 15470,\quad14206,\quad12976,\quad11559,\quad10653.
\]
For every real vector $r$, completing the square gives
\[
 r^tQr\geq\frac{r_i^2}{(Q^{-1})_{ii}}.
\]
If $r\notin[-5,5]^5$, it follows that
\begin{equation}\label{eq:enumeration-bound}
 r^tQr\geq\frac{36\cdot130217}{15470}
      =\frac{2343906}{7735}>252.
\end{equation}
There are $14815$ vectors of norm at most $252$ in the box.

We next explain the power-series elimination.  Permute rows and columns
so that a maximal nonsingular submatrix of the constant coefficient
occurs in the upper left corner.  Write the resulting matrix as
\[
 \begin{pmatrix}B_{11}&B_{12}\\B_{21}&B_{22}\end{pmatrix},
 \qquad\det B_{11}(0)\ne0.
\]
Eliminating $B_{11}$ leaves the Schur complement
$B_{22}-B_{21}B_{11}^{-1}B_{12}$.
If its first $a$ coefficients vanish, divide by $t^a$ and repeat.
These operations are performed over $\Ff_{101}[[t]]$.  Applied to
$\overline A$, they give the following table.  The exponent is the
accumulated exponent at the current stage.
\begin{center}
\begin{tabular}{c|rrrrr}
\toprule
Exponent & $0$&$1$&$2$&$3$&$5$\\
New pivots & $504$&$13$&$7$&$2$&$1$\\
Remaining columns & $24$&$11$&$4$&$2$&$1$\\
\bottomrule
\end{tabular}
\end{center}
Thus $\rk\overline A=527$.  The last column is identically zero by
\eqref{eq:jet-rank-upper}, reduced modulo $101$.  This conclusion is
not inferred from a finite string of zero coefficients.  The pivot
exponents have sum
\[
 13+2\cdot7+3\cdot2+5=38.
\]
In particular, the elimination identifies a nonzero $527\times527$
minor of valuation $38$.

The largest pivot exponent is $5$.  The input through $t^{14}$ therefore
determines the normalized kernel $\overline K(t)$ modulo $t^{10}$.
In the ordering specified by the program, its final coordinate is $1$.
This is the coordinate indexed by $\mu=(2,2,2,2,2)$.
Substitution in the original matrix verifies
\[
 \overline A(t)\overline K(t)\equiv0\pmod{t^{10}}.
\]
The all-orders kernel whose truncation is used here is determined by
the preceding rank identity and the invertible blocks in the elimination.

For a vector of coefficients $w=(w_\mu)$, define the integral
degree-ten polynomial
\begin{equation}\label{eq:tangent-moment}
 \mathcal T(w;t,z)
   =\sum_\mu w_\mu
        \sum_{r\equiv\mu\pmod4}
           (r^tz)^{10}t^{(r^tQr-d_\mu)/8}.
\end{equation}
Up to the common factor $(2\pi i)^{10}/10!$, this is the degree-ten
Taylor term of the corresponding even section.  Restrict it to the
plane $z=U^tx$, where $x=(x_1,x_2,x_3)^t$ and
\begin{equation}\label{eq:plane-matrix}
 U=\begin{pmatrix}
 16&2&4&5&4\\
 16&17&12&1&2\\
 7&9&12&10&6
 \end{pmatrix}.
\end{equation}
The first three columns have determinant $1420$, which is nonzero
modulo $101$.  This matrix therefore defines a plane over both fields.

Let $P(x)=\mathcal T(\overline K;t,U^tx)$ over $\Ff_{101}((t))$.
Its three first partial derivatives have degree $9$.  Consider the map
\begin{equation}\label{eq:macaulay-map}
 \bigoplus_{i=1}^3\Ff_{101}((t))[x_1,x_2,x_3]_{16}
 \longrightarrow \Ff_{101}((t))[x_1,x_2,x_3]_{25},\qquad
 (a_1,a_2,a_3)\longmapsto\sum_{i=1}^3a_i\frac{\partial P}{\partial x_i}.
\end{equation}
The two dimensions are $459$ and $351$.  The matrix, with source
monomials as rows, has the following elimination profile:
\begin{center}
\begin{tabular}{c|rrrrrrrrrr}
\toprule
Exponent & $0$&$1$&$2$&$3$&$4$&$5$&$6$&$7$&$8$&$9$\\
New pivots & $221$&$27$&$76$&$7$&$10$&$4$&$2$&$0$&$3$&$1$\\
\bottomrule
\end{tabular}
\end{center}
The sum is $351$, so \eqref{eq:macaulay-map} is surjective.  The
valuation of the nonzero minor obtained in this elimination is $305$;
the largest pivot exponent is $9$.  Knowledge modulo $t^{10}$ is
sufficient: at each step it determines the next constant pivot block,
and no unknown term can affect the last block at exponent $9$.

Surjectivity means that every monomial of degree $25$ belongs to the
ideal generated by the three partial derivatives of $P$.  These
derivatives cannot have a common projective zero.  Since $101\nmid10$,
the plane curve $P=0$ is smooth.  Every $351\times351$ minor of
\eqref{eq:macaulay-map} is homogeneous of degree $351$ in the coordinates
of a kernel vector.  Lemma~\ref{lem:cofactor-lift}, applied to
\eqref{eq:jet-integral-matrix}, consequently gives a characteristic-zero
kernel vector whose degree-ten form has a smooth plane section.

\begin{remark}
The only global input in the arithmetic lifting is the exact rank
identity \eqref{eq:jet-rank-upper}.  Finite vanishing tests alone would
not supply this identity.  Conversely, no rational lift of the
normalized vector $\overline K(t)$ is needed: the integral cofactor
vector in Lemma~\ref{lem:cofactor-lift} lifts its projective class,
which is all that is required for the tangent-cone calculation.
\end{remark}

\section*{Disclosure of AI Assistance}
Under the author's guidance, OpenAI's large language model Astra was used
extensively to develop the mathematical arguments and proposed proofs and to
prepare the initial draft of this manuscript. The author has reviewed the overall
structure and coherence of the arguments and takes full responsibility for the
content. However, independent line-by-line verification of every step has not yet
been completed. This manuscript records the author's exploration and its
findings. It is shared in its current form to invite feedback and collaboration
from interested readers.

\end{document}